\documentclass[a4paper,11pt,reqno]{amsart}
\usepackage[T1]{fontenc}
\usepackage[utf8]{inputenc}
\usepackage[english]{babel}
\usepackage{amsmath,amsthm,amssymb}
\usepackage{xcolor}
\usepackage{fullpage}
\usepackage{siunitx}
\usepackage{pgfplotstable}
\pgfplotsset{
compat = 1.17,
tick label style = {font = \tiny},
legend style = {font = \tiny},
xlabel style={yshift=+0.5ex},
ylabel style={yshift=-1.0ex}
}
\usepackage[width=0.95\textwidth]{caption}
\usepackage[width=0.95\textwidth,labelfont=rm]{subcaption}
\usepackage{hyperref}
\usepackage{fancyhdr}
\advance\footskip0.5cm
\makeatletter
\def\@seccntformat#1{%
  \protect\textup{\protect\@secnumfont
    \ifnum\pdfstrcmp{subsection}{#1}=0 \bfseries\fi
    \csname the#1\endcsname
    \protect\@secnumpunct
  }%
}
\makeatother
\newtheorem{theorem}{Theorem}[section]
\newtheorem{proposition}[theorem]{Proposition}
\newtheorem{algorithm}[theorem]{Algorithm}
\newtheorem{definition}[theorem]{Definition}
\newtheorem{remark}[theorem]{Remark}
\newcommand\ee{\boldsymbol{e}}
\newcommand\hh{\boldsymbol{h}}
\newcommand\mm{\boldsymbol{m}}
\newcommand\nnu{\boldsymbol{\nu}}
\newcommand\rr{\boldsymbol{r}}
\newcommand\uu{\boldsymbol{u}}
\newcommand\vv{\boldsymbol{v}}
\newcommand\ww{\boldsymbol{w}}
\newcommand\zz{\boldsymbol{z}}
\newcommand\CC{\boldsymbol{C}}
\newcommand\HH{\boldsymbol{H}}
\newcommand\LL{\boldsymbol{L}}

\newcommand\E{\mathcal{E}}
\newcommand\J{\mathcal{J}}
\newcommand\T{\mathcal{T}}
\newcommand\OO{\mathcal{O}}
\renewcommand\S{\mathcal{S}}
\newcommand\Kh{\boldsymbol{\mathcal{K}}_h}
\newcommand\Mh{\boldsymbol{\mathcal{M}}_h}
\newcommand\Nh{\mathcal{N}_h}
\newcommand\pphi{\boldsymbol{\phi}}
\newcommand\vvphi{\boldsymbol{\vphi}}
\newcommand\ppsi{\boldsymbol{\psi}}
\newcommand\vphi{\varphi}
\newcommand\zzeta{\boldsymbol{\zeta}}
\newcommand\interp{\mathcal{I}_h}
\newcommand\Interp{\boldsymbol{\mathcal{I}}_h}
\newcommand\0{\boldsymbol{0}}
\newcommand\sphere{\mathbb{S}^2}
\newcommand\grad{\nabla}
\newcommand\Grad{\boldsymbol{\nabla}}
\newcommand\Lapl{\boldsymbol{\Delta}}
\newcommand\Ph{\mathbb{P}_h}
\newcommand\R{\mathbb{R}}
\newcommand\N{\mathbb{N}}
\renewcommand{\vec}[1]{\mathbf{#1}}
\newcommand{\abs}[1]{\lvert #1 \rvert}
\newcommand{\dual}[3][]{\langle #2,#3 \rangle_{#1}}
\newcommand{\inner}[3][]{\langle #2,#3 \rangle_{#1}}
\newcommand{\norm}[2][]{\lVert #2 \rVert_{#1}}
\DeclareMathOperator{\diam}{diam}

\newcommand\ddt{\frac{\mathrm{d}}{\mathrm{d}t}}
\newcommand\de{\partial}
\newcommand\dt{\mathrm{d}t}
\newcommand\mmt{\de_t \mm}
\newcommand\mmtt{\de_{tt} \mm}

\newcommand\vvt{\de_t \vv}
\newcommand\wwt{\de_t \ww}

\newcommand\eps{\varepsilon}
\newcommand\Ms{M_{\mathrm{s}}}
\newcommand\heff{\hh_{\mathrm{eff}}}
\newcommand\Heff{\HH_{\mathrm{eff}}}
\newcommand\Hext{\HH_{\mathrm{ext}}}
\newcommand{\weakstarto}{\overset{\ast}{\rightharpoonup}}
\newcommand{\weakto}{\rightharpoonup}
\newcommand\hmin{h_{\mathrm{min}}}
\newcommand\Cinv{C_{\mathrm{inv}}}
\begin{document}
%%%%%%%%%%%%%%%%%%%%
\title{Numerical analysis of the Landau--Lifshitz--Gilbert equation with inertial effects}
%%%%%%%%%%%%%%%%%%%%
\author{Michele~Ruggeri}
\address{TU Wien,
Institute of Analysis and Scientific Computing,
Wiedner Hauptstrasse~8--10, 1040 Vienna, Austria}
\email{michele.ruggeri@asc.tuwien.ac.at}
%%%%%%%%%%%%%%%%%%%%
\date{\today}
%%%%%%%%%%%%%%%%%%%%
%%%%%%%%%%%%%%%%%%%%
%%%%%%%%%%%%%%%%%%%%

%%%%%%%%%%%%%%%%%%%%
\begin{abstract}
We consider the numerical approximation of the inertial Landau--Lifshitz--Gilbert equation (iLLG),
which describes the dynamics of the magnetization in ferromagnetic materials at subpicosecond time scales.
We propose and analyze two fully discrete numerical schemes:
The first method is based on a reformulation of the problem
as a linear constrained variational formulation for the linear velocity.
The second method exploits a reformulation of the problem as a first order system in time for
the magnetization and the angular momentum.
Both schemes are implicit, based on first-order finite elements, and generate approximations
satisfying the unit-length constraint of iLLG at the vertices of the underlying mesh.
For both methods,
we prove convergence of the approximations towards a weak solution of the problem.
Numerical experiments validate the theoretical results
and show the applicability of the methods for the simulation of ultrafast magnetic processes.
\end{abstract}
%%%%%%%%%%%%%%%%%%%%

%%%%%%%%%%%%%%%%%%%%
\maketitle
%%%%%%%%%%%%%%%%%%%%

\thispagestyle{fancy}

%%%%%%%%%%%%%%%%%%%%
\section{Introduction}
%%%%%%%%%%%%%%%%%%%%

%%%%%%%%%%%%%%%%%%%%
\subsection{Magnetization dynamics with inertial effects}
%%%%%%%%%%%%%%%%%%%%

The understanding of the magnetization dynamics
and the capability to perform reliable numerical simulations of magnetic systems
play
a fundamental role in the design of many technological applications,
e.g., hard disk drives.
A well-accepted model to describe the magnetization dynamics
in ferromagnetic materials is 
the Landau--Lifshitz--Gilbert equation (LLG),
which, in the so-called Gilbert form, is given by
\begin{equation} \label{eq:LLG}
\mmt = - \gamma_0 \, \mm \times \Heff[\mm] + \alpha \, \mm \times \mmt.
\end{equation}
Here,
$\mm$ denotes the normalized magnetization (dimensionless and satisfying $\abs{\mm}=1$),
the effective field $\Heff[\mm]$ (in \si{\ampere\per\meter}),
up to a negative multiplicative constant,
is the functional derivative of the micromagnetic energy $\E[\mm]$ (in \si{\joule}) with respect to the magnetization,
while $\gamma_0>0$ and $\alpha>0$ denote the gyromagnetic ratio (in \si{\meter\per\ampere\per\second})
and the Gilbert damping parameter (dimensionless), respectively.
The first term on the right-hand side of~\eqref{eq:LLG}
describes the precession of the magnetization around the effective field.
The second term is dissipative and pushes the magnetization
towards the effective field.
The resulting dynamics is a damped precession,
where the magnetization rotates around the effective field while being damped towards it;
see Figure~\ref{fig:llg-vs-illg}(a).

%%%%%%%%%%%%%%%%%%%%
\begin{figure}[h]
\centering
\begin{subfigure}{0.3\textwidth}
\centering
\includegraphics[height=2.7cm]{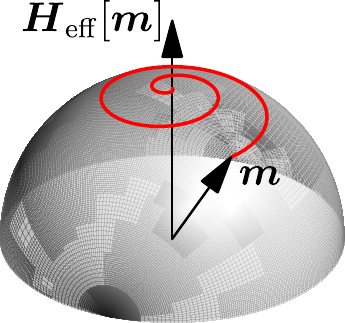}
\subcaption{LLG.}
\end{subfigure}
\quad
\begin{subfigure}{0.3\textwidth}
\centering
\includegraphics[height=2.7cm]{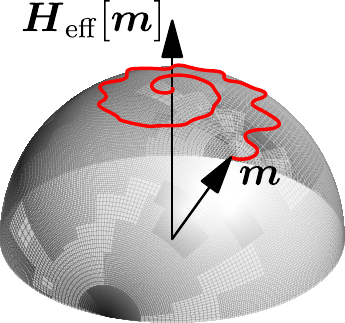}
\subcaption{iLLG.}
\end{subfigure}
\caption{
Schematic of the magnetization dynamics:
(a) LLG (precession and damping);
(b) iLLG (precession, damping, and nutation).}
\label{fig:llg-vs-illg}
\end{figure}
%%%%%%%%%%%%%%%%%%%%

In 1996, a pioneering experiment showed that,
using femtosecond laser excitations, it is possible to manipulate the magnetization
of a nickel sample at subpicosecond time scales~\cite{bmdb1996}.
This discovery gave impulse to several theoretical and experimental studies,
which gave rise to the field that nowadays is referred to as \emph{ultrafast magnetism}~\cite{wm2016}.

The standard LLG~\eqref{eq:LLG} is not capable to describe
the dynamics of the magnetization at such short time scales.
Based on the concept of angular momentum in magnetic spin systems,
a novel evolution equation has been recently proposed~\cite{crw2011}.
This equation, called \emph{inertial} LLG (iLLG), reads as
\begin{equation} \label{eq:LLG-new}
\mmt = - \gamma_0 \, \mm \times \Heff[\mm]
+ \alpha \, \mm \times \mmt
+ \tau \, \mm \times \mmtt.
\end{equation}
Here, $\tau>0$ denotes the angular momentum relaxation time (in~\si{\second}).
In addition to the classical precession and damping contributions,
the right-hand side of~\eqref{eq:LLG-new} comprises a third term involving the second time derivative of the magnetization.
It has been predicted that the effect of this additional contribution on the magnetization dynamics
consists in the appearance of nutation dynamics---superimposed magnetization oscillations occurring
at a frequency much higher than the one of the damped precession dynamics; see Figure~\ref{fig:llg-vs-illg}(b).
Such inertial dynamics has been experimentally observed
for the first time only very recently~\cite{inertia2020}.

%%%%%%%%%%%%%%%%%%%%
\subsection{Numerical approximation of LLG and wave map equation}
%%%%%%%%%%%%%%%%%%%%

This work is concerned with the numerical analysis of~\eqref{eq:LLG-new}.
As LLG has some similarities with
the harmonic map heat flow into the sphere~\cite{lw2008}
\begin{equation*}
\partial_{t} \uu - \Lapl\uu
= \abs{\Grad\uu}^2 \uu,
\end{equation*}
it turns out that iLLG is related to the wave map equation into the sphere~\cite{tataru2004}
\begin{equation} \label{eq:wavemap}
\partial_{tt} \uu - \Lapl\uu
= \big( \abs{\Grad\uu}^2 - \abs{\partial_t \uu}^2 \big) \uu.
\end{equation}
The numerical approximation of this class of partial differential equations (PDEs) poses several challenges:
nonuniqueness, possible blow-up in time, and low regularity of weak solutions,
geometric nonlinearities,
a nonconvex pointwise unit-length constraint,
intrinsic energy laws
as well as (for LLG)
the possible coupling with other PDEs, e.g., the Maxwell equations.

In the last twenty years,
several numerical integrators have been proposed.
Without claiming to be exhaustive and, in particular, restricting ourselves to the methods that are akin
to the ones proposed in the present work,
we refer to the works~\cite{bp2006,aj2006,bkp2008,alouges2008a,kw2018} for LLG,
and to~\cite{bfp2007,blp2009,bartels2009,kw2014,bartels2015a,bartels2016}
for the wave map equation.

%%%%%%%%%%%%%%%%%%%%
\subsection{Contributions and outline of the present work}
%%%%%%%%%%%%%%%%%%%%

In this work,
combining techniques developed for LLG and the wave map equation,
we introduce, analyze, and numerically compare two fully discrete numerical schemes for iLLG.
For both methods, the spatial discretization is based on first-order finite elements.
The first scheme (Algorithm~\ref{alg:tps}) is an extension of the tangent plane scheme proposed
for~\eqref{eq:LLG} in~\cite{alouges2008a}.
The scheme is based on an equivalent reformulation of~\eqref{eq:LLG-new} in the tangent space.
The unit-length constraint is enforced at the vertices of the mesh 
by projecting to the sphere the nodal values of the computed approximation at each time-step.
The second method (Algorithm~\ref{alg:mps})
extends to iLLG
the constraint-preserving angular momentum method proposed
for~\eqref{eq:wavemap} in~\cite{kw2014}.
Following~\cite{bartels2015a}, the spatial discretization
based on finite differences
considered in~\cite{kw2014}
is replaced
by a mass-lumped finite element approximation.
Since the resulting method leads to the solution of a nonlinear system of equations per time-step,
a linearization based on a convergent constraint-preserving fixed-point iteration---similar to those considered in~\cite{bp2006,bartels2015a}
for the methods proposed therein---is discussed and analyzed
(Algorithm~\ref{alg:mpsx}).

We study well-posedness and stability of the proposed schemes,
and determine sufficient conditions which guarantee that the algorithms
satisfy discrete energy laws resembling the one of the continuous problem
(see~\eqref{eq:decay-new} below).
Moreover,
we prove that they generate sequences of finite element solutions that,
upon extraction of a subsequence, converge towards a weak solution of the problem.
The proof is constructive and provides an alternative proof of existence of weak solutions to~\eqref{eq:LLG-new}
(first established in~\cite{ht2014}).
The numerical analysis of iLLG has been considered so far only in~\cite{mt2016},
where a semi-implicit method has been proposed and its conservation properties have been analyzed.
The present work thus proposes the first numerical schemes that are proven to be convergent towards
a weak solution of iLLG.

The remainder of the work is organized as follows:
We conclude this section by collecting some useful notation used throughout the paper.
In Section~\ref{sec:model}, we present the mathematical model under consideration in detail.
In Section~\ref{sec:algorithm}, we introduce the proposed numerical schemes
and state the main results of the work.
Section~\ref{sec:numerics} is devoted to numerical experiments.
Finally, in Section~\ref{sec:proof},
we collect the proofs of the results presented in the paper.

%%%%%%%%%%%%%%%%%%%%
\subsection{Notation} \label{sec:notation}
%%%%%%%%%%%%%%%%%%%%
We denote by $\N = \{ 1,2,\dots\}$ the set of natural numbers
and set $\N_0 := \N \cup \{ 0 \}$.
We denote the unit sphere by $\sphere = \{ x \in \R^3 : \abs{x} = 1 \}$.
We denote by $\{\ee_1, \ee_2, \ee_3\}$ the standard basis of $\R^3$.
For (spaces of) vector-valued or matrix-valued functions, we use bold letters,
e.g., for a generic domain $\Omega \subset \R^d$ ($d=2,3$),
we denote both $L^2(\Omega;\R^3)$ and $L^2(\Omega;\R^{3 \times 3})$ by $\LL^2(\Omega)$.
We denote by $\inner{\cdot}{\cdot}$
both the scalar product of $\LL^2(\Omega)$ and the duality pairing between $\HH^1(\Omega)$ and its dual,
with the ambiguity being resolved by the arguments.
The set of sphere-valued functions in $\HH^1(\Omega)$ is denoted by $H^1(\Omega; \sphere)$.
We also use the notation $\lesssim$ to denote
\emph{smaller than or equal to up to a multiplicative constant},
i.e., we write $A \lesssim B$ if there exists a constant $c>0$,
which is clear from the context and always independent of the discretization parameters,
such that $A \leq c B$.
Finally, we write $A \simeq B$ if $A \lesssim B$ and $B \lesssim A$ hold simultaneously.

%%%%%%%%%%%%%%%%%%%%
\section{Mathematical model} \label{sec:model}
%%%%%%%%%%%%%%%%%%%%

Let $\Omega\subset\R^d$ ($d=2,3$) be a bounded Lipschitz domain.
The energy of $\mm \in H^1(\Omega ; \sphere)$ is described by the Dirichlet energy functional
\begin{equation} \label{eq:llg:energy}
\E[\mm]
=
\frac{1}{2} \norm[\LL^2(\Omega)]{\Grad\mm}^2.
\end{equation}
Minimizers $\mm \in H^1(\Omega ; \sphere)$ of~\eqref{eq:llg:energy}
satisfy the Euler--Lagrange equations
\begin{equation*}
\dual{\heff[\mm]}{\pphi} = 0
\quad
\text{for all } \pphi \in \HH^1(\Omega) \text{ such that } \mm\cdot\pphi = 0 \text{ a.e.\ in } \Omega,
\end{equation*}
which, in strong form, take the form
\begin{subequations} \label{eq:brown}
\begin{alignat}{2}
\label{eq:brown1}
\mm \times \heff[\mm] &= \0
&\quad& \text{in } \Omega,\\
\label{eq:bc}
\partial_{\nnu}\mm &= \0
&\quad& \text{on } \partial\Omega,
\end{alignat}
\end{subequations}
where $\nnu: \partial \Omega \to \sphere$ denote the outward-pointing unit normal vector to $\partial\Omega$.
Here, $\heff[\mm]$ is defined as the opposite of the G\^{a}teaux derivative of the energy, i.e.,
\begin{equation} \label{eq:heff}
- \dual{\heff[\mm]}{\pphi}
= \Big\langle\frac{\delta\E[\mm]}{\delta \mm} , \pphi \Big\rangle
\stackrel{\eqref{eq:llg:energy}}{=} \inner{\Grad\mm}{\Grad\pphi}.
\end{equation}
Nonequilibrium magnetization configurations $\mm(t) \in H^1(\Omega ; \sphere)$
evolve according to LLG (see~\eqref{eq:LLG}), which in rescaled form reads as
\begin{equation} \label{eq:llg}
\mmt = - \mm \times ( \heff[\mm] - \alpha \, \mmt)
\quad
\quad \text{in } \Omega,
\end{equation}
where $\alpha>0$.
Note that stationary solutions to~\eqref{eq:llg} satisfy~\eqref{eq:brown1}.
A simple formal computation reveals the orthogonality
$\mm\cdot\mmt = 0$,
from which it follows that the dynamics inherently preserves the unit-length constraint.
Moreover, taking the scalar product of~\eqref{eq:llg} with $\heff[\mm] - \alpha \, \mmt$,
one can show that any sufficiently smooth solution of~\eqref{eq:llg} satisfies the energy law
\begin{equation} \label{eq:decay}
\ddt\E[\mm(t)] = - \alpha \norm[\LL^2(\Omega)]{\mmt(t)}^2 \leq 0
\quad
\text{for all } t > 0.
\end{equation}
Hence, the dynamics is dissipative with the dissipation being modulated
by the parameter $\alpha$.

Inertial effects can be included in the model by adding a term on the right-hand side of~\eqref{eq:llg} (see~\eqref{eq:LLG-new}).
The resulting equation, iLLG, is given by
\begin{equation} \label{eq:llg-new}
\mmt = - \mm \times ( \heff[\mm] - \alpha \, \mmt - \tau \, \mmtt)
\quad
\quad \text{in } \Omega,
\end{equation}
where $\tau > 0$.
Let $\vv := \mmt$ and $\ww := \mm \times \mmt = \mm \times \vv$.
Using the jargon of kinematics, we refer to these two quantities as
\emph{linear velocity} and \emph{angular momentum}, respectively.
By construction, $\mm$, $\vv$, and $\ww$ are mutually orthogonal.
Moreover, since $\abs{\mm}=1$ and $\mm\cdot\vv = 0$, it holds that $\abs{\vv} = \abs{\ww}$.
The same computation leading to~\eqref{eq:decay}
yields the energy law of iLLG:
\begin{equation} \label{eq:decay-new}
\ddt \left(
\E[\mm(t)] + \frac{\tau}{2} \norm[\LL^2(\Omega)]{\mmt(t)}^2
\right)
= - \alpha \norm[\LL^2(\Omega)]{\mmt(t)}^2 \leq 0
\quad
\text{for all } t > 0.
\end{equation}
This motivates the definition of the extended energy functional
\begin{equation} \label{eq:functional}
\J(\mm,\uu)
= \E[\mm]
+ \frac{\tau}{2} \norm[\LL^2(\Omega)]{\uu}^2
\quad
\text{for all }
\mm \in H^1(\Omega ; \sphere) \text{ and } \uu \in \LL^2(\Omega).
\end{equation}
With this definition,
the quantity decaying over time
in the dynamics governed by~\eqref{eq:llg-new}
becomes $\J(\mm,\mmt)$.
Staying within the framework of kinematics,
we can interpret $\J(\mm,\mmt)$ as the \emph{total energy} of the magnetization,
which comprises the \emph{potential energy} $\E[\mm]$
and the \emph{kinetic energy} $\tau\norm[\LL^2(\Omega)]{\mmt}^2/2$.

The initial boundary value problem considered in this work
consists of~\eqref{eq:llg-new}
supplemented with
homogeneous Neumann boundary conditions~\eqref{eq:bc},
which make the dynamic problem compatible with the stationary case~\eqref{eq:brown},
and suitable initial conditions
$\mm(0) = \mm^0$
and $\mmt(0) = \vv^0$.

We conclude this section by presenting the definition of a weak solution of~\eqref{eq:llg-new},
which is obtained by extending~\cite[Definition~1.2]{as1992} to the present setting;
see also~\cite{ht2014}.

%%%%%%%%%%%%%%%%%%%%
\begin{definition} \label{def:weak}
Let $\mm^0 \in H^1(\Omega ; \sphere)$ and $\vv^0 \in \LL^2(\Omega)$
such that $\mm^0\cdot\vv^0=0$ a.e.\ in $\Omega$.
A vector field $\mm:\Omega \times (0,\infty) \to \sphere$ is called
a \emph{global weak solution} of iLLG~\eqref{eq:llg-new}
if $\mm \in L^{\infty}(0,\infty; \HH^1(\Omega)) \cap W^{1,\infty}(0,\infty; \LL^2(\Omega))$
and, for all $T>0$, the following properties are satisfied:
\begin{itemize}
\item[\rm(i)]  $\mm \in \HH^1(\Omega_T)$, where $\Omega_T := \Omega \times (0,T)$;
\item[\rm(ii)] $\mm(t) \to \mm^0$ in $\HH^1(\Omega)$ and $\mmt(t) \to \vv^0$ in $\LL^2(\Omega)$
as $t \to 0$;
\item[\rm(iii)] For all $\vvphi\in C^{\infty}_c([0,T);\HH^1(\Omega))$, it holds that
\begin{equation} \label{eq:weak:variational}
\begin{split}
& \int_0^T \inner{\mmt(t)}{\vvphi(t)} \, \dt \\
& \quad = - \int_0^T \inner{\heff[\mm(t)]}{\vvphi(t) \times \mm(t)} \, \dt
+ \alpha \int_0^T \inner{\mm(t)\times\mmt(t)}{\vvphi(t)} \, \dt \\
& \qquad - \tau \int_0^T \inner{\mm(t)\times\mmt(t)}{\de_t\vvphi(t)} \, \dt
- \tau \inner{\mm^0\times\vv^0}{\vvphi(0)};
\end{split}
\end{equation}
\item[\rm(iv)] It holds that
\begin{equation} \label{eq:weak:energy}
\J(\mm(T),\mmt(T)) + \alpha \int_0^T \norm[\LL^2(\Omega)]{\mmt(t)}^2 \dt
\leq \J(\mm^0,\vv^0).
\end{equation}
\end{itemize}
\end{definition}
%%%%%%%%%%%%%%%%%%%%

In Definition~\ref{def:weak},
\eqref{eq:weak:variational} comes from a variational formulation of~\eqref{eq:llg-new}
in the space-time cylinder $\Omega_T$,
where we integrate by parts in time the inertial contribution
in order to lower the requested regularity in time of $\mm$.
The energy inequality~\eqref{eq:weak:energy} is the weak counterpart of~\eqref{eq:decay-new}.

%%%%%%%%%%%%%%%%%%%%
\begin{remark}
{\rm(i)}
The setting discussed in this section can be obtained from the original equations
expressed in physical units
after a suitable rescaling.
Let $t$ and $x$ denote the time and spatial variables (measured in \si{\second} and \si{\meter}, respectively).
First, we perform the change of variables
$t' = \gamma_0 \Ms t$ and $x' = x / \ell_{\mathrm{ex}}$,
where $\Ms>0$ and $\ell_{\mathrm{ex}}>0$ denote the saturation magnetization (in~\si{\ampere\per\meter})
and the exchange length (in~\si{\meter}) of the material, respectively.
Second, we rescale the energy $\E[\mm]$ (in~\si{\joule}),
the effective field $\Heff[\mm]$ (in~\si{\ampere\per\meter}),
and the angular momentum relaxation time $\tau$ (in~\si{\second}) in the following way:
$\E'[\mm] = \E[\mm] / (\mu_0 \Ms^2 \ell_{\mathrm{ex}}^3)$,
$\heff[\mm] = \Heff[\mm]/ \Ms$,
$\tau' = \gamma_0 \Ms \tau$.
Then, using the chain rule, \eqref{eq:LLG-new} can be rewritten as~\eqref{eq:llg-new},
where all `primes' are omitted from the rescaled quantities in order to simplify the notation.\\
{\rm(ii)}
For ease of presentation, in the micromagnetic energy functional~\eqref{eq:llg:energy},
we consider only the leading-order exchange contribution.
The numerical treatment of standard lower-order energy contributions
(e.g., magnetocrystalline anisotropy, Zeeman energy, magnetostatic energy, Dzyaloshinskii--Moriya interaction)
is well understood; see, e.g.,
\cite{bffgpprs2014,prs2018,dpprs2019,hpprss2019}.
\end{remark}
%%%%%%%%%%%%%%%%%%%%

%%%%%%%%%%%%%%%%%%%%
\section{Numerical algorithms and main results} \label{sec:algorithm}
%%%%%%%%%%%%%%%%%%%%

In this section, we introduce two fully discrete algorithms for the numerical approximation of iLLG
and we state the corresponding stability and convergence results.

%%%%%%%%%%%%%%%%%%%%
\subsection{Preliminaries}
%%%%%%%%%%%%%%%%%%%%

For the time discretization, 
we consider a uniform partition of the positive real axis $(0,\infty)$ 
with time-step size $k>0$, i.e., $t_i := ik$ for all $i \in \N_0$.
Given a sequence $\{ \phi^i\}_{i \in \N_0}$,
for all $i \in \N_0$, we define
$d_t \phi^{i+1} := (\phi^{i+1} - \phi^i)/k$
and $\phi^{i+1/2} := (\phi^{i+1} + \phi^i)/2$.
Interpreting the sequence $\{ \phi^i\}_{i \in \N_0}$ as a collection of snapshots of a time-dependent function,
we consider the time reconstructions $\phi_{k}$, $\phi^-_{k}$, $\overline{\phi}_{k}$, $\phi^+_{k}$
defined, for all $i \in \N_0$ and $t \in [t_i,t_{i+1})$, as
\begin{equation} \label{eq:reconstructions}
\begin{split}
& \phi_{k}(t) := \frac{t-t_i}{k}\phi^{i+1} + \frac{t_{i+1} - t}{k}\phi^i,
\quad
\phi_{k}^-(t) := \phi^i,
\quad
\overline{\phi}_{k}(t) := \phi^{i+1/2},
\quad
\text{and}
\quad
\phi_{k}^+(t) := \phi^{i+1}.
\end{split}
\end{equation}
Note that $\partial_t \phi_k (t) = d_t \phi^{i+1}$ for all $i \in \N_0$ and $t \in [t_i,t_{i+1})$.

For the spatial discretization, we assume $\Omega$ to be a polytopal domain with Lipschitz boundary
and consider a shape-regular family $\{ \T_h \}_{h>0}$ of tetrahedral meshes of $\Omega$ 
parametrized by the mesh size $h = \max_{K \in \T_h} \diam(K)$.
Moreover, let $\hmin = \min_{K \in \T_h} \diam(K)$.
We denote by $\Nh$ the set of vertices of $\T_h$.
For any $K \in \T_h$, let $\mathcal{P}^1(K)$ be the space of first-order polynomials on $K$.
We denote by $\S^1(\T_h)$ the space of piecewise affine and globally continuous functions
from $\Omega$ to $\R$, i.e.,
$\S^1(\T_h) = \left\{v_h \in C^0(\overline{\Omega}): v_h \vert_K \in \mathcal{P}^1(K) \text{ for all } K \in \T_h \right\}$.
Its classical basis is given
by the set of the nodal hat functions $\left\{\vphi_{z}\right\}_{z\in\Nh}$,
which satisfy $\vphi_{z}(z')=\delta_{z,z'}$ for all $z,z'\in\Nh$.
Let $\interp: C^0(\overline{\Omega}) \to \S^1(\T_h)$
denote the nodal interpolant 
defined by $\interp[v] = \sum_{z \in \Nh} v(z) \vphi_{z}$ for all $v \in C^0(\overline{\Omega})$.
We denote by $\Interp: \CC^0(\overline{\Omega}) \to \S^1(\T_h)^3$ its vector-valued counterpart.
We consider the mass-lumped product $\inner[h]{\cdot}{\cdot}$ defined by
\begin{equation} \label{eq:mass-lumping}
\inner[h]{\ppsi}{\pphi}
= \int_\Omega \interp[\ppsi \cdot \pphi]
\quad \text{for all } \ppsi, \pphi \in \CC^0(\overline{\Omega}).
\end{equation}
Moreover, we define the mapping $\Ph : \HH^1(\Omega)^{\star} \to \S^1(\T_h)^3$ by
\begin{equation} \label{eq:pseudo-projection}
\inner[h]{\Ph \uu}{\pphi_h}
= \dual{\uu}{\pphi_h}
\quad \text{for all } \uu \in \HH^1(\Omega)^{\star} \text{ and } \pphi_h \in \S^1(\T_h)^3.
\end{equation}
Finally, we say that a mesh satisfies the \emph{angle condition}
if all off-diagonal entries of the so-called stiffness matrix are nonpositive, i.e.,
\begin{equation} \label{eq:angleCondition}
\inner{\grad\vphi_{z}}{\grad\vphi_{z'}} \leq 0
\quad \text{for all } z, z' \in \Nh \text{ with } z \neq z'.
\end{equation}
A sufficient condition for~\eqref{eq:angleCondition} to hold in 3D is that the measure
of all dihedral angles of all tetrahedra of the mesh is smaller than or equal to $\pi/2$~\cite{bartels2005}.

%%%%%%%%%%%%%%%%%%%%
\subsection{Numerical algorithms}
%%%%%%%%%%%%%%%%%%%%

In the following algorithms,
the main identities
satisfied by any solution $\mm$ of LLG/iLLG, i.e.,
$\abs{\mm} = 1$ and $\mm\cdot\mmt = 0$,
are imposed only at the vertices of the mesh $\T_h$.
To this end,
we define the \emph{set of admissible discrete magnetizations}
\begin{equation*}
\Mh := \left\{\pphi_h \in \S^1(\T_h)^3: \abs{\pphi_h(z)}=1 \text{ for all } z \in \Nh \right\}
\end{equation*}
and, for $\ppsi_h \in \S^1(\T_h)^3$,
the \emph{discrete tangent space} of $\ppsi_h$
\begin{equation} \label{eq:discreteTangentSpace}
\Kh[\ppsi_h]
:= \left\{
\pphi_h \in \S^1(\T_h)^3 : \ppsi_h(z) \cdot \pphi_h(z) = 0 \text{ for all } z \in \Nh
\right\}.
\end{equation}
The discrete counterpart of the functional~\eqref{eq:functional}
is defined by
\begin{equation*}
\J_h(\mm_h,\uu_h)
= \E[\mm_h]
+ \frac{\tau}{2} \norm[h]{\uu_h}^2
\quad
\text{for all }
\mm_h, \uu_h \in \S^1(\T_h)^3.
\end{equation*}

%%%%%%%%%%%%%%%%%%%%
\subsubsection{Tangent plane scheme}
%%%%%%%%%%%%%%%%%%%%

The first method uses the linear velocity $\vv = \partial_t \mm$ as an auxiliary variable.
Using the vector identity
\begin{equation} \label{eq:triple}
\vec{a}\times(\vec{b}\times\vec{c})=(\vec{a}\cdot\vec{c})\vec{b} - (\vec{a}\cdot\vec{b})\vec{c} \quad \text{for all } \vec{a},\vec{b},\vec{c} \in\R^3,
\end{equation}
together with the properties $\abs{\mm} =1$ and $\mm\cdot\vv = 0$,
iLLG can be formally rewritten as
\begin{equation} \label{eq:llg-alternative}
\tau \,\vvt
+ \alpha \, \vv
+ \mm\times\vv
= \heff[\mm]
- (\heff[\mm]\cdot\mm)\mm
- \tau \abs{\vv}^2 \mm.
\end{equation}
Following the tangent plane paradigm~\cite{aj2006,bkp2008,alouges2008a,akst2014},
to obtain a numerical scheme for iLLG,
we consider a finite element approximation of a mass-lumped variational formulation of~\eqref{eq:llg-alternative}
based on test functions fulfilling the same orthogonality property satisfied by $\vv$.
This yields a natural linearization of~\eqref{eq:llg-alternative},
as the contributions associated with last two (nonlinear) terms on the right-hand side
vanish by orthogonality.
More precisely, for all time-steps $i \in \N_0$,
given the current approximations $\mm_h^i \approx \mm(t_i)$ and $\vv_h^i \approx \vv(t_i)$,
we compute $\vv_h^{i+1} \approx \vv(t_{i+1})$
using a discretized version of~\eqref{eq:llg-alternative},
which is
based on the discrete tangent space $\Kh[\mm_h^i]$ introduced in~\eqref{eq:discreteTangentSpace}
for the spatial discretization
and on the backward Euler method
for the temporal discretization.
Then, using the available approximations $\mm_h^i$ and $\vv_h^{i+1}$,
we obtain $\mm_h^{i+1} \approx \mm(t_{i+1})$ via a first-order time-stepping,
i.e., $\Interp\big[(\mm_h^i + k \vv_h^{i+1})/\abs{\mm_h^i + k \vv_h^{i+1}}\big]$,
where the nodal projection is employed to ensure that the new approximation belongs to $\Mh$.
Unlike~\cite{aj2006,alouges2008a,akst2014} and as in~\cite{bkp2008},
we use the mass-lumped product $\inner[h]{\cdot}{\cdot}$,
which enhances the efficiency of the scheme without affecting its formal convergence order.
The resulting scheme is summarized in the following algorithm.

%%%%%%%%%%%%%%%%%%%%
\begin{algorithm}[tangent plane scheme] \label{alg:tps}
\emph{Input:}
$\mm_h^0 \in \Mh$
and
$\vv_h^0 \in \Kh[\mm_h^0]$.\\
\emph{Loop:}
For all $i \in \N_0$, iterate {\rm(i)--(ii)}:
\begin{itemize}
\item[\rm(i)] Compute $\vv_h^{i+1} \in \Kh[\mm_h^i]$ such that,
for all $\pphi_h \in \Kh[\mm_h^i]$, it holds that
\begin{equation} \label{eq:tps1}
\begin{split}
& \tau \inner[h]{d_t\vv_h^{i+1}}{\pphi_h}
+ \alpha \inner[h]{\vv_h^{i+1}}{\pphi_h}
+ \inner[h]{\mm_h^i\times\vv_h^{i+1}}{\pphi_h}
- k \inner[h]{\Ph\heff[\vv_h^{i+1}]}{\pphi_h} \\
& \quad = \inner[h]{\Ph\heff[\mm_h^i]}{\pphi_h}.
\end{split}
\end{equation}
\item[\rm(ii)] Define $\mm_h^{i+1} = \Interp\big[(\mm_h^i + k \vv_h^{i+1})/\abs{\mm_h^i + k \vv_h^{i+1}}\big] \in \Mh$.
\end{itemize}
\emph{Output:}
Sequence of approximations $\left\{(\mm_h^{i+1},\vv_h^{i+1})\right\}_{i \in \N_0}$.
\end{algorithm}
%%%%%%%%%%%%%%%%%%%%

Algorithm~\ref{alg:tps} is well-defined:
The (nonsymmetric) bilinear form on the left-hand side of~\eqref{eq:tps1} in step~{\rm(i)} is elliptic,
so that existence and uniqueness of a solution $\vv_h^{i+1} \in \Kh[\mm_h^i]$
are guaranteed by the Lax--Milgram theorem;
In the nodal projection appearing in step~{\rm(ii)},
the denominator is bounded from below by 1,
so that division by zero never occurs.

%%%%%%%%%%%%%%%%%%%%
\subsubsection{Angular momentum method}
%%%%%%%%%%%%%%%%%%%%

The second method, following~\cite{kw2014,bartels2015a}, uses the angular momentum
$\ww = \mm \times \partial_t \mm$ as an auxiliary variable.
First,
note that
\begin{equation*}
\mm \times \ww = \mm \times (\mm \times \mmt) = - \mmt,
\end{equation*}
where the second identity follows from~\eqref{eq:triple},
together with $\abs{\mm} =1$ and $\mm\cdot\mmt = 0$.
Then,
we have that
\begin{equation*}
\tau \, \wwt
=
\tau \, \mm \times \mmtt
\stackrel{\eqref{eq:llg-new}}{=}
\mm \times \heff[\mm] - \alpha \, \mm \times \mmt + \mmt.
\end{equation*}
We infer the first-order (in time) system
\begin{align*}
\mmt &= - \mm \times \ww, \\
\tau \, \wwt &= \mm \times \heff[\mm] - \alpha \, \mm \times \mmt - \mm \times \ww.
\end{align*}
For all $i \in \N_0$,
given $\mm_h^i \approx \mm(t_i)$ and $\ww_h^i \approx \ww(t_i)$,
approximations $\mm_h^{i+1} \approx \mm(t_{i+1})$ and $\ww_h^{i+1} \approx \ww(t_{i+1})$
are computed by solving a mass-lumped variational formulation
of this first-order system,
where the time discretization is based on the midpoint rule.
The resulting scheme is stated in the following algorithm.

%%%%%%%%%%%%%%%%%%%%
\begin{algorithm}[nonlinear angular momentum method] \label{alg:mps}
\emph{Input:}
$\mm_h^0 \in \Mh$
and
$\vv_h^0 \in \Kh[\mm_h^0]$.\\
\emph{Initialization:}
Define $\ww_h^0 = \Interp\big[\mm_h^0 \times \vv_h^0\big] \in \Kh[\mm_h^0]$.\\
\emph{Loop:}
For all $i \in \N_0$, compute $(\mm_h^{i+1} ,  \ww_h^{i+1}) \in \Mh \times \Kh[\mm_h^{i+1}]$ such that,
for all $(\pphi_h , \ppsi_h) \in \S^1(\T_h)^3 \times \S^1(\T_h)^3$, it holds that
\begin{subequations} \label{eq:mps}
\begin{align}
\label{eq:mps1}
\inner[h]{d_t\mm_h^{i+1}}{\pphi_h}
& =
- \inner[h]{\mm_h^{i+1/2} \times \ww_h^{i+1/2}}{\pphi_h}, \\
\tau \inner[h]{d_t\ww_h^{i+1}}{\ppsi_h}
& =
\inner[h]{\mm_h^{i+1/2} \times \Ph\heff[\mm_h^{i+1/2}]}{\ppsi_h} \notag \\
& \qquad - \alpha \inner[h]{\mm_h^{i+1/2} \times d_t\mm_h^{i+1} }{\ppsi_h}
- \inner[h]{\mm_h^{i+1/2} \times \ww_h^{i+1/2}}{\ppsi_h}.
\label{eq:mps2}
\end{align}
\end{subequations}
\emph{Output:}
Sequence of approximations $\left\{(\mm_h^{i+1},\ww_h^{i+1})\right\}_{i \in \N_0}$.
\end{algorithm}
%%%%%%%%%%%%%%%%%%%%

In the following proposition,
we show that
the pointwise constraints
$\abs{\mm}=1$ and $\mm\cdot\ww = 0$
are inherently preserved 
by Algorithm~\ref{alg:mps}
(at the vertices of the mesh).
Its proof is deferred to Section~\ref{sec:solvability}.

%%%%%%%%%%%%%%%%%%%%
\begin{proposition} \label{prop:conservation}
Let $i \in \N_0$.
The approximations generated by Algorithm~\ref{alg:mps} satisfy
$\mm_h^{i+1} \in \Mh$
and 
$\ww_h^{i+1} \in \Kh[\mm_h^{i+1}]$.
\end{proposition}
%%%%%%%%%%%%%%%%%%%%

The computation of $(\mm_h^{i+1} ,  \ww_h^{i+1})$ satisfying~\eqref{eq:mps}
involves the solution of a nonlinear system of equations.
An effective implementation requires a linearization.

Let $\uu_h^i := \mm_h^{i+1/2}$ and $\zz_h^i := \ww_h^{i+1/2}$.
Performing simple algebraic manipulations,
we rewrite~\eqref{eq:mps} with respect to the unknowns $\uu_h^i$ and $\zz_h^i$:
\begin{subequations} \label{eq:mpsx}
\begin{align}
2\inner[h]{\uu_h^i}{\pphi_h}
+ k \inner[h]{\uu_h^i \times \zz_h^i}{\pphi_h}
& =
2\inner[h]{\mm_h^i}{\pphi_h}, \\
2 \tau \inner[h]{\zz_h^i}{\ppsi_h}
- k \inner[h]{\uu_h^i \times \Ph\heff[\uu_h^i]}{\ppsi_h} \qquad \notag \\
- 2 \alpha \inner[h]{\uu_h^i \times \mm_h^i}{\ppsi_h}
+ k \inner[h]{\uu_h^i \times \zz_h^i}{\ppsi_h}
& =
2 \tau \inner[h]{\ww_h^i}{\ppsi_h}.
\end{align}
\end{subequations}
Starting from this formulation,
in the following algorithm
we introduce a linear fixed-point iteration
(similar in spirit to those considered in~\cite{bp2006,bartels2015a}),
which provides an effective implementation of Algorithm~\ref{alg:mps}.

%%%%%%%%%%%%%%%%%%%%
\begin{algorithm}[linearized angular momentum method] \label{alg:mpsx}
\emph{Input:}
$\mm_h^0 \in \Mh$
and
$\vv_h^0 \in \Kh[\mm_h^0]$.\\
\emph{Initialization:}
Define $\ww_h^0 = \Interp\big[\mm_h^0 \times \vv_h^0\big] \in \Kh[\mm_h^0]$.\\
\emph{Loop:}
For all $i \in \N_0$, iterate {\rm(i)--(ii)}:
\begin{itemize}
\item[\rm(i)]
Let $\uu_h^{i,0} = \mm_h^i$ and $\zz_h^{i,0} = \ww_h^i$.
For all $\ell \in \N_0$, iterate {\rm(i-a)--(i-b)}:
\begin{itemize}
\item[\rm(i-a)]
Compute $\uu_h^{i,\ell+1} \in \S^1(\T_h)^3$ such that,
for all $\pphi_h \in \S^1(\T_h)^3$, it holds that
\begin{equation} \label{eq:mpsx1}
2\inner[h]{\uu_h^{i,\ell+1}}{\pphi_h}
+ k \inner[h]{\uu_h^{i,\ell+1} \times \zz_h^{i,\ell}}{\pphi_h}
=
2\inner[h]{\mm_h^i}{\pphi_h};
\end{equation}
\item[\rm(i-b)]
Compute $\zz_h^{i,\ell+1} \in \S^1(\T_h)^3$ such that,
for all $\ppsi_h \in \S^1(\T_h)^3$, it holds that
\begin{equation} \label{eq:mpsx2}
\begin{split}
& 2 \tau \inner[h]{\zz_h^{i,\ell+1}}{\ppsi_h}
+ k \inner[h]{\uu_h^{i,\ell+1} \times \zz_h^{i,\ell+1}}{\ppsi_h} \\
& \quad =
k \inner[h]{\uu_h^{i,\ell+1} \times \Ph\heff[\uu_h^{i,\ell+1}]}{\ppsi_h}
+ 2 \alpha \inner[h]{\uu_h^{i,\ell+1} \times \mm_h^i}{\ppsi_h}
+ 2 \tau \inner[h]{\ww_h^i}{\ppsi_h};
\end{split}
\end{equation}
\end{itemize}
until
\begin{equation} \label{eq:mpsx_stopping}
\norm[h]{\uu_h^{i,\ell+1} - \uu_h^{i,\ell}}
+ \norm[h]{\zz_h^{i,\ell+1} - \zz_h^{i,\ell}} \le \eps.
\end{equation}
\item[\rm(ii)]
Let $\ell_i \in \N_0$ be the smallest integer for which the stopping criterion~\eqref{eq:mpsx_stopping} is met.
Define $\mm_h^{i+1} := 2 \uu_h^{i,\ell_i+1} - \mm_h^i$
and $\ww_h^{i+1} := 2 \zz_h^{i,\ell_i+1} - \ww_h^i$.
\end{itemize}
\emph{Output:}
Sequence of approximations $\left\{(\mm_h^{i+1},\ww_h^{i+1})\right\}_{i \in \N_0}$.
\end{algorithm}
%%%%%%%%%%%%%%%%%%%%

The well-posedness and the conservation properties of Algorithm~\ref{alg:mpsx}
are the subject of the following proposition.
Its proof is postponed to Section~\ref{sec:solvability}.

%%%%%%%%%%%%%%%%%%%%
\begin{proposition} \label{prop:fixed-point}
Let $i \in \N_0$.
Suppose that $\mm_h^i \in \Mh$.\\
{\rm(i)}
For all $\ell \in \N_0$, \eqref{eq:mpsx1} and~\eqref{eq:mpsx2} admit unique solutions
$\uu_h^{i,\ell+1}$ and $\zz_h^{i,\ell+1}$ in $\S^1(\T_h)^3$. 
Moreover, it holds that $\norm[\LL^{\infty}(\Omega)]{\uu_h^{i,\ell+1}} \leq 1$.\\
{\rm(ii)}
There exist $k_0>0$ and $C>0$
such that,
if $k < k_0$ and $k < C \hmin$,
then,
for all $\ell \in \N_0$,
it holds that
\begin{equation} \label{eq:contraction} 
\norm[h]{\uu_h^{i,\ell+2}-\uu_h^{i,\ell+1}}
+
\norm[h]{\zz_h^{i,\ell+2}-\zz_h^{i,\ell+1}}
\le
q
\big(
\norm[h]{\uu_h^{i,\ell+1} - \uu_h^{i,\ell}}
+ \norm[h]{\zz_h^{i,\ell+1} - \zz_h^{i,\ell}}
\big)
\end{equation}
for some $0 < q < 1$.
The constants $k_0$, $C$, and $q$ depend only on the shape-regularity of $\T_h$ and the problem data.\\
{\rm(iii)}
Under the assumptions of part~{\rm(ii)},
the stopping criterion~\eqref{eq:mpsx_stopping} is met in a finite number of iterations.
If $\ell_i \in \N_0$ denotes the smallest integer for which~\eqref{eq:mpsx_stopping} holds,
the new approximations $\mm_h^{i+1} = 2 \uu_h^{i,\ell_i+1} - \mm_h^i$
and $\ww_h^{i+1} = 2 \zz_h^{i,\ell_i+1} - \ww_h^i$
belong to $\Mh$ and $\Kh[\mm_h^{i+1}]$, respectively.
\end{proposition}
%%%%%%%%%%%%%%%%%%%%

Proposition~\ref{prop:fixed-point}{\rm(ii)} shows that,
under suitable assumptions,
the mapping defining the fixed-point iteration is a contraction.
Therefore, under the same assumptions,
Banach fixed-point theorem ensures that~\eqref{eq:mps}
admits a unique solution so that Algorithm~\ref{alg:mps} is well-posed.

In view of the stability and convergence analysis,
we observe that,
for all $i \in \N_0$,
the iterates $(\mm_h^{i+1} ,  \ww_h^{i+1}) \in \Mh \times \Kh[\mm_h^{i+1}]$ of Algorithm~\ref{alg:mpsx}
satisfy
\begin{align*}
\inner[h]{d_t\mm_h^{i+1}}{\pphi_h}
& =
- \inner[h]{\mm_h^{i+1/2} \times \ww_h^{i+1/2}}{\pphi_h}
+ \inner[h]{\mm_h^{i+1/2} \times \rr_h^i}{\pphi_h}, \\
\tau \inner[h]{d_t\ww_h^{i+1}}{\ppsi_h}
& =
\inner[h]{\mm_h^{i+1/2} \times \Ph\heff[\mm_h^{i+1/2}]}{\ppsi_h}
- \alpha \inner[h]{\mm_h^{i+1/2} \times d_t\mm_h^{i+1}}{\ppsi_h} \\
& \qquad
- \inner[h]{\mm_h^{i+1/2} \times \ww_h^{i+1/2}}{\ppsi_h}
\end{align*}
for all $(\pphi_h , \ppsi_h) \in \S^1(\T_h)^3 \times \S^1(\T_h)^3$,
where $\rr_h^i := \zz_h^{i,\ell_i+1} - \zz_h^{i,\ell_i} \in \S^1(\T_h)^3$
satisfies $\norm[h]{\rr_h^i} \le \eps$.

%%%%%%%%%%%%%%%%%%%%
\subsection{Stability and convergence results}
%%%%%%%%%%%%%%%%%%%%

In the following proposition, we establish the discrete energy laws satisfied by the algorithms.
Its proof is postponed to Section~\ref{sec:stability}.

%%%%%%%%%%%%%%%%%%%%
\begin{proposition}[Discrete energy law and stability] \label{prop:energy}
Let $i \in \N_0$.\\
{\rm(i)}
Suppose that the mesh $\T_h$ satisfies angle condition~\eqref{eq:angleCondition}.
The approximations generated by Algorithm~\ref{alg:tps} satisfy the discrete energy law
\begin{equation} \label{eq:tps:energy}
\J_h(\mm_h^{i+1},\vv_h^{i+1})
+ \alpha k \norm[h]{\vv_h^{i+1}}^2
+ \frac{\tau k^2}{2} \norm[h]{d_t \vv_h^{i+1}}^2
+  \frac{k^2}{2} \norm[\LL^2(\Omega)]{\Grad\vv_h^{i+1}}^2
\leq \J_h(\mm_h^i,\vv_h^i).
\end{equation}
{\rm(ii)}
The approximations generated by Algorithm~\ref{alg:mps} satisfy the discrete energy law
\begin{equation} \label{eq:mps:energy}
\J_h(\mm_h^{i+1},\ww_h^{i+1})
+ \alpha k \norm[h]{d_t\mm_h^{i+1}}^2
= \J_h(\mm_h^i,\ww_h^i).
\end{equation}
{\rm(iii)}
The approximations generated by Algorithm~\ref{alg:mpsx} satisfy the discrete energy law
\begin{equation} \label{eq:mpsx:energy}
\begin{split}
\J_h(\mm_h^{i+1},\ww_h^{i+1})
+ \alpha k \norm[h]{d_t\mm_h^{i+1}}^2
+ k \inner[h]{\mm_h^{i+1/2} \times \rr_h^i}{\Ph\heff[\mm_h^{i+1/2}] - \alpha & \, d_t\mm_h^{i+1}} \\
& = \J_h(\mm_h^i,\ww_h^i).
\end{split}
\end{equation}
\end{proposition}
%%%%%%%%%%%%%%%%%%%%

In the energy law satisfied by Algorithm~\ref{alg:tps},
besides the LLG-intrinsic dissipation,
we observe the presence of numerical dissipation due to the use of the backward Euler method;
cf.\ the last two terms on the left-hand side of~\eqref{eq:tps:energy}.
Algorithm~\ref{alg:mps} fulfills a discrete energy identity,
which reflects the fact that the midpoint rule is symplectic.
The same identity,
apart from an additional term coming from
the inexact solution of the nonlinear system,
is satisfied by Algorithm~\ref{alg:mpsx}.

From each algorithm, we obtain a sequence of approximations
$\{\mm_h^i\}_{i \in \N_0}$, which we can use to
define the piecewise affine time reconstruction $\mm_{hk} : (0,\infty) \to \S^1(\T_h)^3$
(denoted by $\mm_{hk}^{\eps}$ in the case of Algorithm~\ref{alg:mpsx})
as
\begin{equation*}
\mm_{hk}(t) := \frac{t-t_i}{k}\mm_h^{i+1} + \frac{t_{i+1} - t}{k}\mm_h^i
\quad
\text{for all }
i \in \N_0
\text{ and }
t \in [t_i,t_{i+1});
\end{equation*}
see~\eqref{eq:reconstructions}.
In the following theorem, we show that, under appropriate assumptions,
the sequence $\{\mm_{hk}\}$ (resp., $\{\mm_{hk}^{\eps}\}$ for Algorithm~\ref{alg:mpsx})
converges in a suitable sense towards solutions of iLLG
as $h$, $k$ (and $\eps$) go to $0$.
Its proof is postponed to Section~\ref{sec:convergence}.

%%%%%%%%%%%%%%%%%%%%
\begin{theorem} \label{thm:convergence}
Let the approximate initial conditions satisfy
\begin{equation} \label{eq:convergence0}
\mm_h^0 \to \mm^0 \quad \text{in } \HH^1(\Omega)
\quad
\text{and}
\quad
\vv_h^0 \to \vv^0 \quad \text{in } \LL^2(\Omega)
\quad
\text{as } h \to 0.
\end{equation}
{\rm(i)}
For Algorithm~\ref{alg:tps},
assume that each mesh $\T_h$ satisfies the angle condition~\eqref{eq:angleCondition}
and that $k = o(\hmin^{d/2})$ as $h,k \to 0$.
For Algorithm~\ref{alg:mps}, assume that the scheme is well-posed.
Then, there exist a global weak solution $\mm : \Omega \times (0,\infty) \to \sphere$ of iLLG
in the sense of Definition~\ref{def:weak}
and a (nonrelabeled) subsequence of $\{ \mm_{hk} \}$
which converges towards $\mm$ as $h,k \to 0$.
In particular, as $h,k \to 0$, it holds that
$\mm_{hk} \weakstarto \mm$ in $L^{\infty}(0,\infty; H^1(\Omega ; \sphere))$
and $\mm_{hk}\vert_{\Omega_T} \weakto \mm\vert_{\Omega_T}$ in $\HH^1(\Omega_T)$ for all $T>0$.\\
{\rm(ii)}
For Algorithm~\ref{alg:mpsx}, assume that the scheme is well-posed and that
$\eps = \OO(\hmin)$ as $h,\eps \to 0$.
For all $T>0$, there exist $\mm : \Omega \times (0,T) \to \sphere$, 
which satisfies the requirements {\rm(i)}--{\rm(iv)} of Definition~\ref{def:weak},
and a (nonrelabeled) subsequence of $\{ \mm_{hk}^\eps \}$
such that 
$\mm_{hk}^{\eps}\vert_{\Omega_T} \weakstarto \mm$ in $L^{\infty}(0,T; H^1(\Omega ; \sphere))$
and $\mm_{hk}^{\eps}\vert_{\Omega_T} \weakto \mm$ in $\HH^1(\Omega_T)$
as $h,k,\eps\to 0$.
\end{theorem}
%%%%%%%%%%%%%%%%%%%%

In order to refer to the limit of the approximations generated by Algorithm~\ref{alg:mpsx},
we have not used the expression `global weak solution'.
This is a consequence of the fact that,
for this algorithm,
the boundedness result
we are able to show 
(see Proposition~\ref{prop:boundedness} below)
is not uniform with respect to the (arbitrary but fixed) final time $T>0$.

To conclude,
we summarize the results of our analysis for the proposed algorithms:
\begin{itemize}
\item
Algorithm~\ref{alg:tps} is unconditionally well-posed, unconditionally stable
(under the angle condition~\eqref{eq:angleCondition}),
and its convergence towards a global weak solution of iLLG
requires the CFL condition $k = o(\hmin^{d/2})$ as $h,k \to 0$.
\item
Algorithm~\ref{alg:mps} is well-posed if
$k$ is sufficiently small and
the CFL condition
$k < C \hmin$ holds.
Assuming its well-posedness,
it is unconditionally stable and unconditionally convergent
towards a global weak solution of iLLG.
\item
Algorithm~\ref{alg:mpsx} is well-posed if
$k$ is sufficiently small and
the CFL condition
$k < C \hmin$ holds.
Assuming its well-posedness and choosing a stopping tolerance $\eps$ having the same order of $\hmin$,
for all $T>0$,
it is unconditionally stable and unconditionally convergent
towards a function
$\mm : \Omega \times (0,T) \to \sphere$ 
which fulfills the properties~{\rm(i)}--{\rm(iv)} of Definition~\ref{def:weak}.
\end{itemize}

%%%%%%%%%%%%%%%%%%%
\begin{remark}
For the sake of brevity,
we restrict ourselves to the case of algorithms generating approximations
which satisfy the unit-length constraint at the vertices of the mesh,
i.e., $\mm_h^i \in \Mh$ for all $i \in \N_0$.
For the tangent plane scheme, this property is guaranteed by the use of the nodal projection.
A theoretical consequence is that
the stability analysis requires the assumption of the angle condition~\eqref{eq:angleCondition},
which turns out to be quite restrictive in 3D. 
For LLG,
a tangent plane scheme which avoids the nodal projection (and the related angle condition)
was proposed in~\cite{ahpprs2014}; see also~\cite{bartels2016,ft2017}.
We believe that a similar approach can be pursued to construct
a projection-free tangent plane scheme for iLLG.
\end{remark}
%%%%%%%%%%%%%%%%%%%

%%%%%%%%%%%%%%%%%%%%
\section{Numerical results} \label{sec:numerics}
%%%%%%%%%%%%%%%%%%%%

Before presenting the proof of the results stated in Section~\ref{sec:algorithm},
we aim to show the effectivity of the proposed algorithms by means of two
numerical experiments.
For the sake of brevity, in this section, we refer to the tangent plane scheme (Algorithm~\ref{alg:tps})
as TPS and to the angular momentum method (Algorithm~\ref{alg:mps} or, more appropriately,
its effective realization given in Algorithm~\ref{alg:mpsx}) as AMM.
The computation presented in this section were obtained with a MATLAB
implementation of the proposed algorithms.
All linear systems were solved using the direct solver provided by MATLAB's backslash operator.

%%%%%%%%%%%%%%%%%%%%
\subsection{Finite-time blow-up of weak solutions}  \label{sec:blow-up}
%%%%%%%%%%%%%%%%%%%%

We investigate the performance
of the algorithms for different choices of the discretization parameters $h$ and $k$ (and $\eps$ for AMM).
At the same time, 
we numerically study for a weak solution $\mm$ of iLLG the occurrence of a so-called \emph{finite-time blow-up},
i.e., whether
there exists $T^*>0$ such that
\begin{equation*}
\lim_{t \to T^*} \norm[\LL^{\infty}(\Omega)]{\Grad\mm(t)} = \infty.
\end{equation*}
To this end, we adapt to iLLG the model problem studied in~\cite{bfp2007,kw2014,bartels2015a}
for the wave map equation and in~\cite{bkp2008} for LLG.

We consider the nondimensional setting presented in Section~\ref{sec:model}
for the unit square domain $\Omega = (-1/2,1/2)^2$ in the time interval $(0,2)$.
We set $\alpha = \tau = 1$ in~\eqref{eq:llg-new}
and consider
homogeneous Neumann boundary conditions~\eqref{eq:bc}
as well as the initial conditions
$\mm^0(x) = (2a(x)x_1, 2a(x)x_2, a(x)^2 - \abs{x}^2) / (a(x)^2 + \abs{x}^2)$
with $a(x) = \max\{ 0 , (1 - 2\abs{x})^4 \}$ for all $x = (x_1,x_2) \in \Omega$
and $\vv^0 \equiv \0$.
For the wave map equation~\eqref{eq:wavemap} and LLG~\eqref{eq:llg},
this setting leads to numerical approximations with large gradients,
which suggests the occurrence of a finite-time blow-up.
For snapshots of numerical approximations which illustrate this phenomenon,
we refer to, e.g., \cite[Figure~3--4]{bfp2007} or~\cite[Figures~1--2]{bkp2008}.

First, we investigate the convergence of the fixed-point iteration in AMM
analyzed in Proposition~\ref{prop:fixed-point}{\rm(ii)}.
For $\ell = 5,6,7$, we consider a uniform mesh $\T_{h_\ell}$ of the unit square
consisting of $2^{2\ell+1}$ rectangular triangles.
The resulting mesh size is $h_\ell = \sqrt{2} \, 2^{-\ell}$.
In the stopping criterion~\eqref{eq:mpsx_stopping},
in order to better evaluate the convergence of the fixed-point iteration,
we use the small tolerance $\eps =$ \num{1e-12}.
The iteration is terminated either when
the stopping criterion~\eqref{eq:mpsx_stopping} is met
or when the number of iterations exceeds \num{1000}.
We use the time-step size $k_\ell = \delta h_\ell$ for different values of $0 < \delta < 1$.

%%%%%%%%%%%%%%%%%%%%
\begin{table}[h]
\begin{tabular}{c|cccc}
& $\delta =$ 0.1 & $\delta =$ 0.2 & $\delta =$ 0.3 & $\delta =$ 0.4 \\
\hline
$\ell = 5$ & \num{4.99} & \num{7.99} & \num{13.87} & \num{35.66}  \\
$\ell = 6$ & \num{4.93} & \num{7.89} & \num{13.82} & \num{35.39}  \\
$\ell = 7$ & \num{4.84} & \num{7.75} & \num{13.66} & \num{34.55}  \\
\end{tabular}
\caption{Experiment of Section~\ref{sec:blow-up}:
Average number of fixed-point iterations
needed to reach the prescribed tolerance $\eps =$ \num{1e-12}
for $\ell= 5,6,7$
and $\delta =$ \num{0.1}, \num{0.2}, \num{0.3}, \num{0.4}.}
\label{tab:fixed-point}
\end{table}
%%%%%%%%%%%%%%%%%%%%

In Table~\ref{tab:fixed-point},
we show the average number of fixed-point iterations needed to
reach the prescribed tolerance $\eps$
for $\delta =$ \num{0.1}, \num{0.2}, \num{0.3}, \num{0.4}.
The number of iterations increases as $\delta$ increases
and decreases (very slightly) as the mesh size decreases.
For all $\ell= 5,6,7$,
the fixed-point iteration
does not converge (within the prescribed maximum number of iterations)
if $\delta$ is larger than a threshold value located between
$0.46$ and $0.47$.
This behavior is in agreement with the dependence
of the contraction constant on the discretization parameters
which can be inferred from the proof of Proposition~\ref{prop:fixed-point}{\rm(ii)},
i.e.,
$q \simeq \delta (1 + h)$
(recall that the tolerance $\eps$ is fixed).

Next,
we compare the performance of TPS and AMM.
We consider the uniform mesh $\T_{h_5}$ (\num{2048} elements and mesh size $h_5=$ \num{0.0442})
and $k = h_5/10$.
For AMM, we set $\eps = h_5/10$ in~\eqref{eq:mpsx_stopping}.

%%%%%%%%%%%%%%%%%%%%
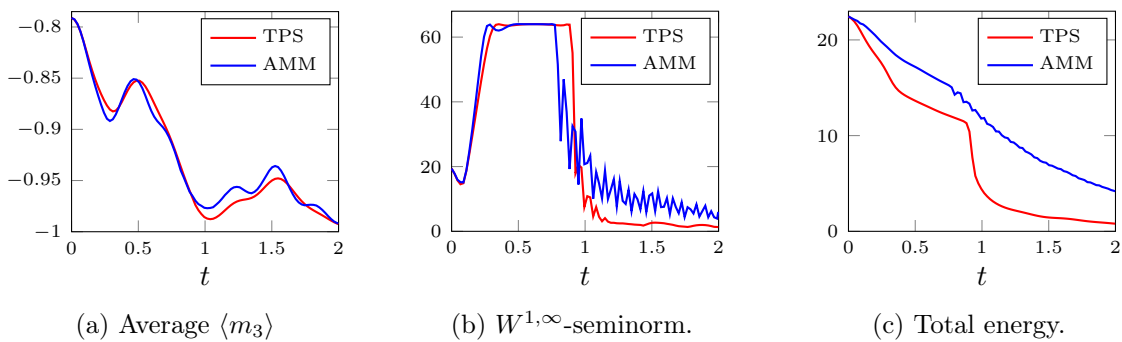
\begin{figure}[h]
\begin{subfigure}{0.32\textwidth}
\centering
\begin{tikzpicture}
\pgfplotstableread{data/exp1/tps_h5_k01_e01.dat}{\tps}
\pgfplotstableread{data/exp1/amm_h5_k01_e01.dat}{\amm}
\begin{axis}
[
width = \textwidth, height=4.5cm,
xlabel = {$t$},
xmin = 0,
xmax = 2,
ymin = -1,
ymax = -0.785,
legend pos= north east,
legend cell align= left,
]
\addplot[red, thick]	table[x=t, y=m3]{\tps};
\addplot[blue, thick]	table[x=t, y=m3]{\amm};
\legend{
TPS,
AMM,
}
\end{axis}
\end{tikzpicture}
\caption{Average $\langle m_3 \rangle$}
\end{subfigure}
%%%
\begin{subfigure}{0.32\textwidth}
\centering
\begin{tikzpicture}
\pgfplotstableread{data/exp1/tps_h5_k01_e01.dat}{\tps}
\pgfplotstableread{data/exp1/amm_h5_k01_e01.dat}{\amm}
\begin{axis}
[
width = \textwidth, height=4.5cm,
xlabel = {$t$},
xmin = 0,
xmax = 2,
ymin = 0,
ymax = 68,
legend pos= north east,
legend cell align= left,
]
\addplot[red, thick]	table[x=t, y=W1infty]{\tps};
\addplot[blue, thick]	table[x=t, y=W1infty]{\amm};
\legend{
TPS,
AMM,
}
\end{axis}
\end{tikzpicture}
\caption{$W^{1,\infty}$-seminorm.}
\end{subfigure}
%%%
\begin{subfigure}{0.32\textwidth}
\centering
\begin{tikzpicture}
\pgfplotstableread{data/exp1/tps_h5_k01_e01.dat}{\tps}
\pgfplotstableread{data/exp1/amm_h5_k01_e01.dat}{\amm}
\begin{axis}
[
width = \textwidth, height=4.5cm,
xlabel = {$t$},
xmin = 0,
xmax = 2,
ymin = 0,
ymax = 23,
legend pos= north east,
legend cell align= left,
]
\addplot[red, thick]	table[x=t, y=tot]{\tps};
\addplot[blue, thick]	table[x=t, y=tot]{\amm};
\legend{
TPS,
AMM,
}
\end{axis}
\end{tikzpicture}
\caption{Total energy.}
\end{subfigure}
\caption{Experiment of Section~\ref{sec:blow-up}:
Comparison of the results obtained with TPS and AMM
for $\ell=5$, $k=h_5/10$, and $\eps=h_5/10$.
(a) Evolution of the average magnetization $\langle m_3 \rangle$.
(b) Evolution of the $W^{1,\infty}$-seminorm.
(c) Evolution of the total energy.}
\label{fig:exp1_blowup}
\end{figure}
%%%%%%%%%%%%%%%%%%%%

In Figure~\ref{fig:exp1_blowup},
we plot the evolutions of
the spatial average of the third magnetization component,
i.e.,
$\langle m_3(t) \rangle := \lvert\Omega\rvert^{-1} \int_{\Omega} \mm_{hk}(t)\cdot\ee_3$,
the $W^{1,\infty}$-seminorm $\norm[\LL^{\infty}(\Omega)]{\Grad\mm_{hk}(t)}$,
and
the total discrete energy $\J_h(\mm_{hk}(t),\partial_t\mm_{hk}(t))$
for $t \in [0,2]$.
We observe that the algorithms capture the same average magnetization dynamics.
In particular, at $t \approx 0.3$,
the approximations attain the largest possible value of the
$W^{1,\infty}$-seminorm for functions in $\Mh$
residing in $\T_{h_5}$, which, for all $\ell=5,6,7$, is given by
\begin{equation} \label{eq:Winfty}
\max_{\pphi_h \in \Mh} \norm[\LL^{\infty}(\Omega)] {\Grad\pphi_h} = 
\max_{\pphi_h \in \Mh} \max_{T \in \T_{h,\ell}} \abs{\Grad\pphi_h\vert_T} =
2 / 2^{-\ell} = 2^{\ell+1},
\end{equation}
which is obtained when the magnetizations of two neighboring vertices point to opposite directions.
Indeed, in our case, the magnetization at $(0,0)$ points to the out-of-plane direction $(1,0,0)$,
while all surrounding vectors point to the opposite direction.
This configuration lasts for some time (see the `plateau' in Figure~\ref{fig:exp1_blowup}(b)).
Then, at $t \approx 0.8$, the magnetization at $(0,0)$ is reversed.
This gives rise to oscillations of decaying amplitude.
Looking at Figure~\ref{fig:exp1_blowup}(c),
we observe that,
in agreement with~\eqref{eq:tps:energy},
the total energy decays monotonically in the case of TPS.
In the case of AMM,
the decay is nonmonotone.
Note that possible lack of monotonicity is predicted by the energy law of AMM;
cf.\ the (unsigned) third term on the left-hand side~\eqref{eq:mpsx:energy}.
Moreover, we see that
the curve for TPS is well below the one of AMM.
This fact can be justified by the numerical dissipation
of the backward Euler method;
cf.\ the second and the third terms on the left-hand side of~\eqref{eq:tps:energy}.

%%%%%%%%%%%%%%%%%%%%
\begin{figure}[ht]
\begin{subfigure}{0.32\textwidth}
\centering
\begin{tikzpicture}
\pgfplotstableread{data/exp1/tps_h5_k001_e001.dat}{\tps}
\pgfplotstableread{data/exp1/amm_h5_k001_e001.dat}{\amm}
\begin{axis}
[
width = \textwidth, height=4.5cm,
xlabel = {$t$},
xmin = 0,
xmax = 2,
ymin = -1,
ymax = -0.785,
legend pos= north east,
legend cell align= left,
]
\addplot[red, thick]	table[x=t, y=m3]{\tps};
\addplot[blue, thick]	table[x=t, y=m3]{\amm};
\legend{
TPS,
AMM,
}
\end{axis}
\end{tikzpicture}
\caption{Average $\langle m_3 \rangle$}
\end{subfigure}
%%%
\begin{subfigure}{0.32\textwidth}
\centering
\begin{tikzpicture}
\pgfplotstableread{data/exp1/tps_h5_k001_e001.dat}{\tps}
\pgfplotstableread{data/exp1/amm_h5_k001_e001.dat}{\amm}
\begin{axis}
[
width = \textwidth, height=4.5cm,
xlabel = {$t$},
xmin = 0,
xmax = 2,
ymin = 0,
ymax = 68,
legend pos= north east,
legend cell align= left,
]
\addplot[red, thick]	table[x=t, y=W1infty]{\tps};
\addplot[blue, thick]	table[x=t, y=W1infty]{\amm};
\legend{
TPS,
AMM,
}
\end{axis}
\end{tikzpicture}
\caption{$W^{1,\infty}$-seminorm.}
\end{subfigure}
%%%
\begin{subfigure}{0.32\textwidth}
\centering
\begin{tikzpicture}
\pgfplotstableread{data/exp1/tps_h5_k001_e001.dat}{\tps}
\pgfplotstableread{data/exp1/amm_h5_k001_e001.dat}{\amm}
\begin{axis}
[
width = \textwidth, height=4.5cm,
xlabel = {$t$},
xmin = 0,
xmax = 2,
ymin = 0,
ymax = 23,
legend pos= north east,
legend cell align= left,
]
\addplot[red, thick]	table[x=t, y=tot]{\tps};
\addplot[blue, thick]	table[x=t, y=tot]{\amm};
\legend{
TPS,
AMM,
}
\end{axis}
\end{tikzpicture}
\caption{Total energy.}
\end{subfigure}
\caption{Experiment of Section~\ref{sec:blow-up}:
Comparison of the results obtained with TPS and AMM
for $\ell=5$, $k=h_5/100$, and $\eps=h_5/100$.
(a) Evolution of the average magnetization $\langle m_3 \rangle$.
(b) Evolution of the $W^{1,\infty}$-seminorm.
(c) Evolution of the total energy.}
\label{fig:exp1_blowup2}
\end{figure}
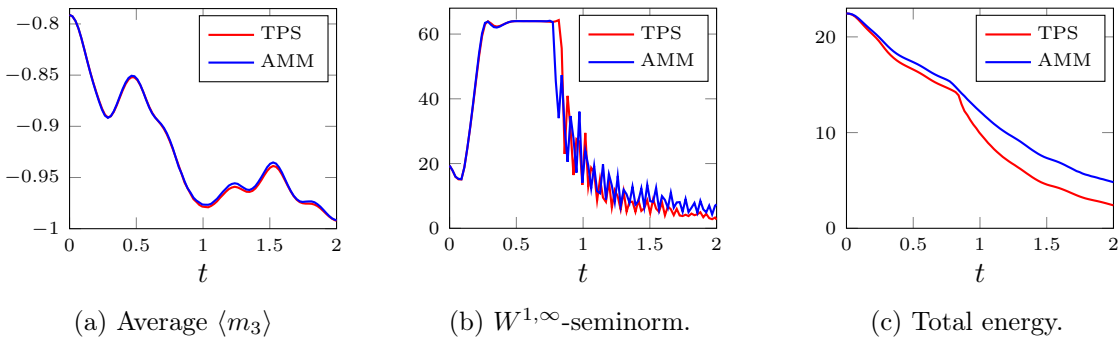
%%%%%%%%%%%%%%%%%%%%

In Figure~\ref{fig:exp1_blowup2},
we show the results obtained repeating the experiment using the same mesh $\T_{h_5}$,
but smaller time-step size $k = h_5/100$ and tolerance $\eps = h_5/100$.
The overall behavior remains the same.
However,
we see that the numerical dissipation of TPS
and the nonmonotonicity of the energy decay of AMM
are reduced.
This observation confirms the validity of the energy laws established in Proposition~\ref{prop:energy},
as the terms responsible for the two above effects
can indeed be controlled by time-step size $k$ and the tolerance $\eps$, respectively.

%%%%%%%%%%%%%%%%%%%%
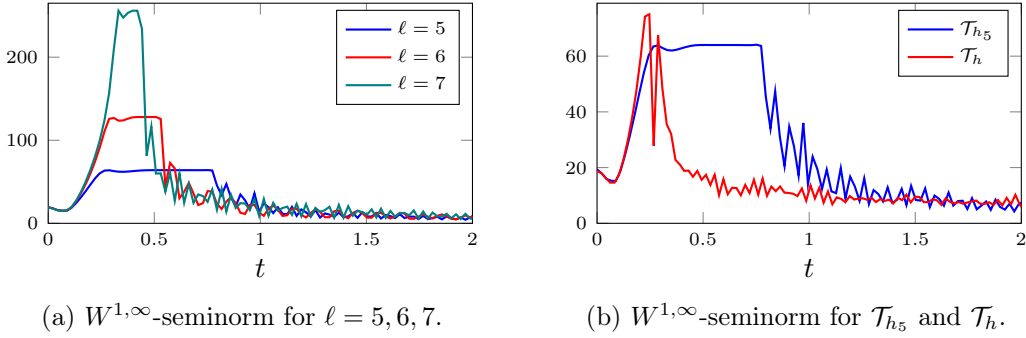
\begin{figure}[ht]
\begin{subfigure}{0.45\textwidth}
\centering
\begin{tikzpicture}
\pgfplotstableread{data/exp1/amm_h5_k001_e001.dat}{\coarse}
\pgfplotstableread{data/exp1/amm_h6_k001_e001.dat}{\medium}
\pgfplotstableread{data/exp1/amm_h7_k001_e001.dat}{\fine}
\begin{axis}
[
width = \textwidth, height=4.5cm,
xlabel = {$t$},
xmin = 0,
xmax = 2,
ymin = 0,
ymax = 265,
legend pos= north east,
legend cell align= left,
]
\addplot[blue, thick]	table[x=t, y=W1infty]{\coarse};
\addplot[red, thick]	table[x=t, y=W1infty]{\medium};
\addplot[teal, thick]	table[x=t, y=W1infty]{\fine};
\legend{
$\ell = 5$,
$\ell = 6$,
$\ell = 7$,
}
\end{axis}
\end{tikzpicture}
\caption{$W^{1,\infty}$-seminorm for $\ell=5,6,7$.}
\end{subfigure}
%%%
\begin{subfigure}{0.45\textwidth}
\centering
\begin{tikzpicture}
\pgfplotstableread{data/exp1/amm_h5_k001_e001.dat}{\struct}
\pgfplotstableread{data/exp1/amm_netgen_k001_e001.dat}{\unstruct}
\begin{axis}
[
width = \textwidth, height=4.5cm,
xlabel = {$t$},
xmin = 0,
xmax = 2,
ymin = 0,
ymax = 79,
legend pos= north east,
legend cell align= left,
]
\addplot[blue, thick]	table[x=t, y=W1infty]{\struct};
\addplot[red, thick]	table[x=t, y=W1infty]{\unstruct};
\legend{
$\T_{h_5}$,
$\T_h$,
}
\end{axis}
\end{tikzpicture}
\caption{$W^{1,\infty}$-seminorm for $\T_{h_5}$ and $\T_h$.}
\end{subfigure}
\caption{Experiment of Section~\ref{sec:blow-up}:
Evolution of the $W^{1,\infty}$-seminorm
obtained with AMM for different mesh resolutions
(a) and different symmetry properties (b).}
\label{fig:exp1_mesh_comparison}
\end{figure}
%%%%%%%%%%%%%%%%%%%%

Finally,
we investigate whether the resolution or the symmetry of the mesh
have an influence on the detection of the blow-up.
In Figure~\ref{fig:exp1_mesh_comparison}(a),
we compare the evolution of the $W^{1,\infty}$-seminorm
obtained using the uniform meshes $\T_{h_\ell}$
for $\ell = 5,6,7$ ($k = h_\ell/100$ and $\eps = h_\ell/100$).
The appearance of blow-up of the three approximations occurs at the same time ($t \approx 0.3$),
but the length of the plateaux,
i.e., the duration of the configuration in which the magnetization of the origin
has the opposite direction of the surrounding vertices,
decreases with the mesh size.
Moreover, for all $\ell=5,6,7$, the maximum value attained by the $W^{1,\infty}$-seminorm
is always the maximum value~\eqref{eq:Winfty} allowed by the discrete space.

In Figure~\ref{fig:exp1_mesh_comparison}(b),
we compare the evolution of the $W^{1,\infty}$-seminorm computed
using $\T_{h_5}$
with the one obtained using an unstructured mesh $\T_h$
of comparable number of elements and mesh size
(\num{2396} elements and $h=$ \num{0.0456}).
We observe that a finite-time blow-up at $t \approx 0.3$ occurs also
for the approximation residing in the unstructured mesh $\T_h$.
However, the magnetization configuration with maximum gradient is
quickly left (no plateau in the evolution of the $W^{1,\infty}$-seminorm).
We believe that the stagnation of the configuration with maximum gradient
observed for uniform meshes is a numerical artifact related to their symmetry.

Since the computations performed with TPS lead to the same conclusions,
in order not to overload the plots,
in Figure~\ref{fig:exp1_mesh_comparison}
we have shown only the results computed using AMM.

It is not clear to the author whether the observed finite-time blow-up
also occurs for the weak solution of iLLG towards which the computed approximations
converge as $h,k,\eps \to 0$.
However, the fact that
the phenomenon has been observed for approximations computed using two different schemes
and various choices of the discretization parameters
seems to provide a clear evidence in this direction.

%%%%%%%%%%%%%%%%%%%%
\subsection{Nutation dynamics in ferromagnetic thin films} \label{sec:nutation}
%%%%%%%%%%%%%%%%%%%%

With this experiment, we aim to illustrate the differences
in the magnetization dynamics induced by the standard LLG~\eqref{eq:LLG} and iLLG~\eqref{eq:LLG-new}.
Moreover, we compare the performance of TPS and AMM in a simulation with physically relevant geometry and material parameters.

The domain $\Omega$ is a planar thin film of the form $\omega \times (0,c)$
with cross section $\omega \subset \R^2$ (parallel to the $x_1 x_2$-plane)
and thickness $c =$ \SI{3}{\nano\meter} (aligned with $\ee_3$).
The cross section $\omega$ is an elliptic domain with semiaxis lengths $a=$ \SI{100}{\nano\meter} and $b=$ \SI{50}{\nano\meter}.
The axes of the ellipse are parallel to $\ee_1$ and $\ee_2$,
with the major axis being parallel to $\ee_1$.
It is well known~\cite{gj1997,carbou2001,difratta2020} that the magnetization of thin films is usually homogeneous in the out-of-plane component,
so that an adequate description of the energetics of the magnet can be obtained using a 2D energy functional $\E[\mm]$
posed only on the cross section $\omega$.
Accordingly, we consider the energy functional
\begin{equation*}
\E[\mm] / c
=
A \int_{\omega} \abs{\Grad\mm}^2
+
K \int_{\omega} [1 - (\mm\cdot\ee_1)^2]
-
\mu_0 \Ms \int_{\omega} \Hext\cdot\mm 
+
\frac{\mu_0 \Ms^2}{2} \int_{\omega} (\mm\cdot\ee_3)^2.
\end{equation*}
Here, $\mu_0 = 4 \pi\cdot$\SI{e-7}{\newton\per\ampere\squared} is the vacuum permeability,
$\Ms$, $A$, $K$ are positive material parameters (see below),
while $\Hext$ denotes an applied magnetic field (in \si{\ampere\per\meter}).
The four energy contributions in $\E[\mm]$ are
exchange interaction, uniaxial anisotropy (with easy axis $\ee_1$ parallel to the major axis of the ellipse),
Zeeman energy, and magnetostatic interaction, respectively.
Note that the nonlocal magnetostatic energy is replaced by a local planar anisotropy contribution penalizing out-of-plane magnetization configurations,
which is admissible for magnetic thin films~\cite{gj1997,carbou2001,difratta2020}.
In~\eqref{eq:LLG}--\eqref{eq:LLG-new},
for the gyromagnetic ratio,
we consider the value
$\gamma_0 =$ \SI{2.211e5}{\meter\per\ampere\per\second},
while the effective field $\Heff[\mm]$ is related to the energy $\E[\mm]$
via the relation $\mu_0 \Ms \, \Heff[\mm] = -\frac{\delta\E[\mm]}{\delta\mm}$.
For the material parameters, we use the values of permalloy (see, e.g., \cite{MUMAG}):
$\Ms =$ \SI{8e5}{\ampere\per\meter},
$A =$ \SI{1.3e-11}{\joule\per\meter},
$K =$ \SI{5e2}{\joule\per\meter\squared},
and $\alpha =$ \num{0.023}.
For the angular momentum relaxation time $\tau$ in~\eqref{eq:LLG-new},
we consider the value
$\tau = \alpha\,\xi$ with $\xi =$ \SI{12.3}{\pico\second}; see~\cite{inertia2020}.
As initial conditions, we consider the constant fields 
$\mm^0 \equiv \ee_1$ and $\vv^0 \equiv \0$.
Note that $\mm^0$ is a global minimum for the energy $\E[\mm]$ if $\Hext \equiv \0$.

%%%%%%%%%%%%%%%%%%%%
\begin{figure}[h]
\begin{subfigure}{0.48\textwidth}
\centering
\begin{tikzpicture}
\begin{axis}
[
width = \textwidth, height=4.5cm,
samples=100,
xlabel = {$t$ [\si{\pico\second}]},
xmin = -0.5,
xmax = 30.5,
]
\addplot+[no markers,thick,domain=0:2,blue] {0.01*sin(2*pi*28.65*x)};
\addplot+[no markers,thick,domain=2:30,blue] {0};
\end{axis}
\end{tikzpicture}
\caption{Pulse field amplitude.}
\end{subfigure}
%%%
\hfill
%%%
\begin{subfigure}{0.48\textwidth}
\centering
\begin{tikzpicture}
\pgfplotstableread{data/exp2/llg_tps_001fs.dat}{\llg}
\pgfplotstableread{data/exp2/illg_tps_001fs.dat}{\illg}
\begin{axis}
[
width = \textwidth, height=4.5cm,
xlabel = {$t$ [\si{\pico\second}]},
xmin = -0.5,
xmax = 30.5,
legend pos= south east,
legend cell align= left,
]
\addplot[blue, thick]	table[x=t, y=m3]{\llg};
\addplot[red, thick]	table[x=t, y=m3]{\illg};
\legend{
LLG,
iLLG,
}
\end{axis}
\end{tikzpicture}
\caption{LLG dynamics vs.\ iLLG dynamics.}
\end{subfigure} \\
%%%
\vspace*{5mm}
%%%
\begin{subfigure}{0.48\textwidth}
\centering
\begin{tikzpicture}
\pgfplotstableread{data/exp2/illg_tps_001fs.dat}{\one}
\pgfplotstableread{data/exp2/illg_tps_010fs.dat}{\ten}
\pgfplotstableread{data/exp2/illg_tps_100fs.dat}{\hundred}
\begin{axis}
[
width = \textwidth, height=4.5cm,
xlabel = {$t$ [\si{\pico\second}]},
xmin = -0.5,
xmax = 30.5,
legend pos= south east,
legend cell align= left,
]
\addplot[blue, thick]	table[x=t, y=m3]{\one};
\addplot[red, thick]	table[x=t, y=m3]{\ten};
\addplot[teal, thick]	table[x=t, y=m3]{\hundred};
\legend{
$\Delta t =$ \SI{1}{\femto\second},
$\Delta t =$ \SI{10}{\femto\second},
$\Delta t =$ \SI{100}{\femto\second},
}
\end{axis}
\end{tikzpicture}
\caption{iLLG dynamics computed with TPS for different time-step sizes.}
\end{subfigure}
%%%
\hfill
%%%
\begin{subfigure}{0.48\textwidth}
\centering
\begin{tikzpicture}
\pgfplotstableread{data/exp2/illg_amm_001fs.dat}{\one}
\pgfplotstableread{data/exp2/illg_amm_010fs.dat}{\ten}
\pgfplotstableread{data/exp2/illg_amm_100fs.dat}{\hundred}
\begin{axis}
[
width = \textwidth, height=4.5cm,
xlabel = {$t$ [\si{\pico\second}]},
xmin = -0.5,
xmax = 30.5,
legend cell align= left,
legend pos= south east,
]
\addplot[blue, thick]	table[x=t, y=m3]{\one};
\addplot[red, thick]	table[x=t, y=m3]{\ten};
\addplot[teal, thick]	table[x=t, y=m3]{\hundred};
\legend{
$\Delta t =$ \SI{1}{\femto\second},
$\Delta t =$ \SI{10}{\femto\second},
$\Delta t =$ \SI{100}{\femto\second},
}
\end{axis}
\end{tikzpicture}
\caption{iLLG dynamics computed with AMM for different time-step sizes.}
\end{subfigure}
\caption{Experiment of Section~\ref{sec:nutation}:
(a) Plot of the function $t \mapsto F(t)$ which modulates the amplitude of the pulse field.
(b) Evolution of $\langle m_3 \rangle$ for LLG and iLLG computed with TPS for $\Delta t =$ \SI{1}{\femto\second}.
(c) Evolution of $\langle m_3 \rangle$ computed with TPS for different time-step sizes.
(d) Evolution of $\langle m_3 \rangle$ computed with AMM for different time-step sizes.
}
\label{fig:nutation}
\end{figure}
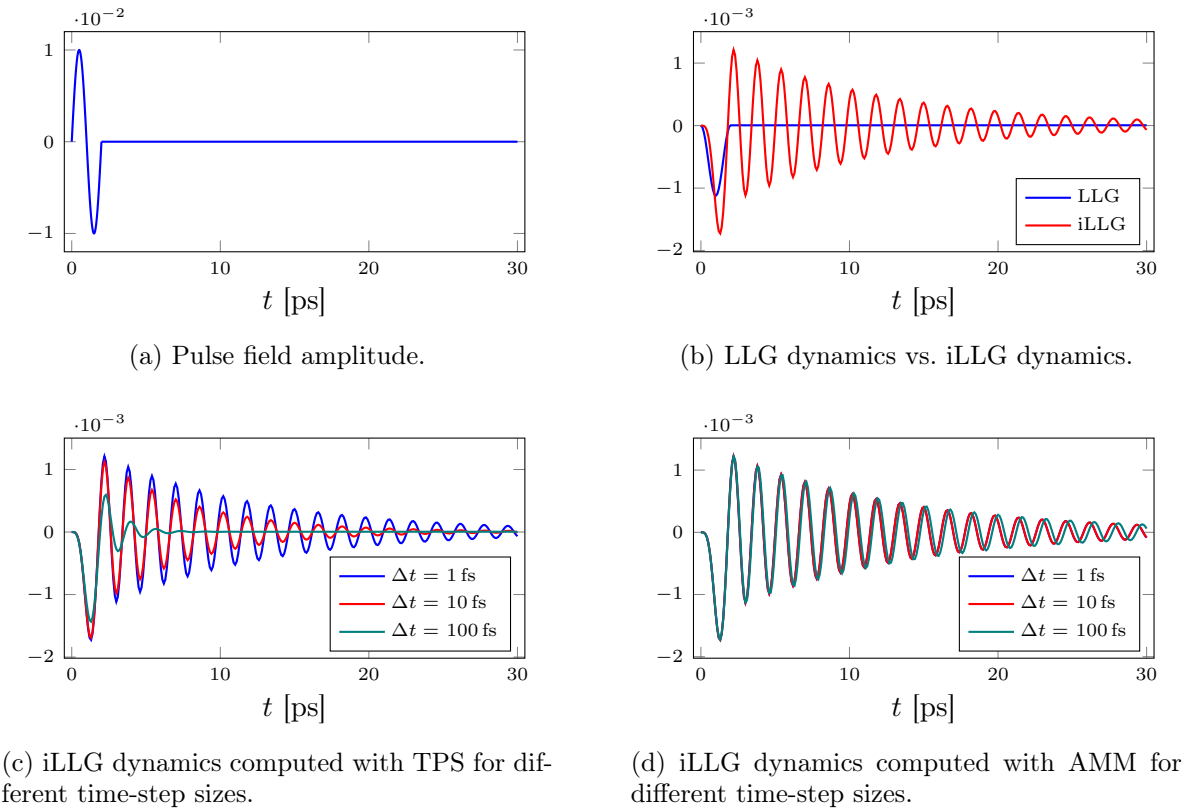
%%%%%%%%%%%%%%%%%%%%

The overall simulation time is \SI{30}{\pico\second}.
The experiment consists in perturbing the equilibrium state $\mm^0 \equiv \ee_1$
with a perpendicular and spatially uniform high-frequency pulse field $\Hext (t) = F(t) \Ms \, \ee_2$,
where $F(t) = 0.01 \, \sin(2 \pi f t) \, \raisebox{1pt}{$\chi$}_{\{ 0 \, \le \, t \, \le \, \num{2e-12} \}}(t)$
with $f =$ \SI{500}{\giga\hertz};
see Figure~\ref{fig:nutation}(a).
In order to assess the resulting magnetization dynamics, we analyze
the time evolution of the spatial average of the third magnetization component
$\langle m_3 \rangle$.

For the spatial discretization we consider a triangular mesh of $\omega$
made of \num{5998} elements.
Its mesh size (\SI{3.760}{\nano\meter}) is well below the exchange length of the material
$\ell_{\mathrm{ex}} = \sqrt{2A / (\mu_0 \Ms^2)} =$ \SI{5.686}{\nano\meter}.
For the time discretization, we consider three different time-step sizes
($\Delta t =$ \num{1}, \num{10}, \SI{100}{\femto\second}).
In the stopping criterion~\eqref{eq:mpsx_stopping}, we use the tolerance $\eps =$ \num{1e-6}.

In Figure~\ref{fig:nutation}(b), we compare the
evolution of $\langle m_3 \rangle$ for LLG and iLLG.
We show the results computed using TPS with $\Delta t =$ \SI{1}{\femto\second}.
Note that TPS for LLG can be obtained from Algorithm~\ref{alg:tps}
by omitting the first term on the left-hand side of~\eqref{eq:tps1};
see~\cite{alouges2008a,bffgpprs2014}.
The dynamics induced by the two models are completely different.
For LLG, the magnetization reacts to the pulse field and returns straight
to the equilibrium state.
For iLLG,
the deflection from the equilibrium state gives rise to oscillations with approximately the
same frequency of the inducing pulse field (\SI{500}{\giga\hertz}).
Due to damping, the amplitude of the oscillations decays with time
and the magnetization regains the initial equilibrium state.
This experiment provides a numerical evidence of the inertial nutation dynamics predicted by the model,
which has been experimentally observed only very recently~\cite{inertia2020}.

In Figure~\ref{fig:nutation}(c), we plot the evolution of $\langle m_3 \rangle$ computed
using TPS with $\Delta t =$ \num{1}, \num{10}, \SI{100}{\femto\second}.
For larger time-step sizes,
we observe a faster decay of the oscillations.
This phenomenon is a consequence of the artificial damping 
of the backward Euler method used for the time discretization.
The observed dependence on $\Delta t$ reflects the fact that the
artificial damping can be controlled by the time-step size; see~\eqref{eq:tps:energy}.
Finally, in Figure~\ref{fig:nutation}(d),
we show the same plot for AMM.
Unlike TPS, AMM is robust with respect to variations of the time-step size.
This reflects the energy conservation properties of the symplectic midpoint rule;
see \eqref{eq:mps:energy}--\eqref{eq:mpsx:energy}.
All the considered time-step sizes are sufficiently small to guarantee
the convergence of the fixed-point iteration,
which requires 1-2 iterations for $\Delta t =$ \num{1}, \SI{10}{\femto\second}
and 2-3 iterations for $\Delta t =$ \SI{100}{\femto\second}.
Note that using larger time-step sizes is not advisable,
as they cannot resolve the pulse field and the resulting magnetization dynamics.

This experiment shows the importance of designing a numerical scheme
which respects the energy law of the underlying model.
This general statement, which holds true for any PDE with a physical background (and, in particular, for LLG),
is a key aspect for iLLG due to the small extent and the ultrafast time scale
of the nutation dynamics.

%%%%%%%%%%%%%%%%%%%%
\section{Proofs} \label{sec:proof}
%%%%%%%%%%%%%%%%%%%%

In this section,
we present the proofs of the results stated in Section~\ref{sec:algorithm}.
In view of their later use, we recall some facts:
the mass-lumped product $\inner[h]{\cdot}{\cdot}$ defined in~\eqref{eq:mass-lumping}
is a scalar product on $\S^1(\T_h)^3$ and
the induced norm $\norm[h]{\cdot}$ satisfies the norm equivalence
\begin{equation*}
\norm[\LL^2(\Omega)]{\pphi_h}
\leq \norm[h]{\pphi_h}
\leq \sqrt{d+2} \, \norm[\LL^2(\Omega)]{\pphi_h}
\quad \text{for all } \pphi_h \in \S^1(\T_h)^3.
\end{equation*}
Moreover, there holds the error estimate
\begin{equation} \label{eq:h-scalar-product}
\lvert \inner{\pphi_h}{\ppsi_h} -  \inner[h]{\pphi_h}{\ppsi_h} \rvert
\le C h^2 \norm[\LL^2(\Omega)]{\Grad\pphi_h} \norm[\LL^2(\Omega)]{\Grad\ppsi_h}
\quad \text{for all } \pphi_h, \ppsi_h \in \S^1(\T_h)^3,
\end{equation}
where $C>0$ depends only on the shape-regularity of $\T_h$;
see~\cite[Lemma~3.9]{bartels2015}.
Finally, we recall
the following relations
between the $L^p$-norm of a discrete function
and the $\ell^p$-norm of the vector collecting its nodal values
(see~\cite[Lemma~3.4]{bartels2015}):
\begin{equation} \label{eq:normEquivalence}
\norm[\LL^p(\Omega)]{\pphi_h}^p \simeq \sum_{z \in \Nh} h_z^d \abs{\pphi_h(z)}^p
\quad
\text{and}
\quad
\norm[\LL^{\infty}(\Omega)]{\pphi_h} = \max_{z \in \Nh} \abs{\pphi_h(z)}
\quad
\text{for all } \pphi_h \in \S^1(\T_h)^3.
\end{equation}
Here, $h_z>0$ denotes the diameter of the nodal patch associated with $z \in \Nh$.

%%%%%%%%%%%%%%%%%%%%
\subsection{Properties of Algorithm~\ref{alg:mps} and well-posedness of Algorithm~\ref{alg:mpsx}}
\label{sec:solvability}
%%%%%%%%%%%%%%%%%%%%

We start with proving the conservation properties of Algorithm~\ref{alg:mps}.

%%%%%%%%%%%%%%%%%%%%
\begin{proof}[Proof of Proposition~\ref{prop:conservation}]
Let $z \in \Nh$.
Choosing $\pphi_h = \vphi_{z} \mm_h^{i+1/2}(z)$
in~\eqref{eq:mps1}, we infer that
$\abs{\mm_h^{i+1}(z)} = \abs{\mm_h^i(z)}$.
Since $\mm_h^0 \in \Mh$ by assumption, we conclude that $\mm_h^{i+1} \in \Mh$.

Choosing $\pphi_h = \vphi_{z} \ww_h^{i+1/2}(z)$
in~\eqref{eq:mps1}
and
$\ppsi_h = \vphi_{z} \mm_h^{i+1/2}(z)$
in~\eqref{eq:mps2}, we obtain the identities
$d_t\mm_h^{i+1}(z) \cdot \ww_h^{i+1/2}(z) = 0$
and
$d_t\ww_h^{i+1}(z) \cdot \mm_h^{i+1/2}(z)$,
respectively.
It follows that
\begin{equation*}
\mm_h^{i+1}(z) \cdot \ww_h^{i+1}(z)
-
\mm_h^i(z) \cdot \ww_h^i(z)
=
d_t\mm_h^{i+1}(z) \cdot \ww_h^{i+1/2}(z)
+
d_t\ww_h^{i+1}(z) \cdot \mm_h^{i+1/2}(z)
= 0.
\end{equation*}
Since
$\mm_h^0 (z) \cdot \ww_h^0 (z) = \mm_h^0 (z) \cdot (\mm_h^0 (z) \times \vv_h^0 (z) ) = 0$,
we conclude that
$\ww_h^{i+1} \in \Kh[\mm_h^{i+1}]$.
\end{proof}
%%%%%%%%%%%%%%%%%%%%

Next, we show that the fixed-point iteration designed for the solution of~\eqref{eq:mpsx}
is well-posed and converges.

%%%%%%%%%%%%%%%%%%%%
\begin{proof}[Proof of Proposition~\ref{prop:fixed-point}]
Let $\ell \in \N_0$.
The bilinear forms on the left-hand side of both~\eqref{eq:mpsx1} and~\eqref{eq:mpsx2}
are elliptic.
Therefore, existence and uniqueness of solutions
$\uu_h^{i,\ell+1}$ and $\zz_h^{i,\ell+1}$ in $\S^1(\T_h)^3$
follow from the Lax--Milgram theorem.

Let $z \in \Nh$ be an arbitrary vertex.
Testing~\eqref{eq:mpsx1} with $\pphi_h = \vphi_{z} \uu_h^{i,\ell+1}(z) \in \S^1(\T_h)^3$,
we obtain that
$\abs{\uu_h^{i,\ell+1}(z)}^2 = \uu_h^{i,\ell+1}(z) \cdot \mm_h^i(z)$.
Hence, $\abs{\uu_h^{i,\ell+1}(z)} \leq \abs{\mm_h^i(z)} = 1$.
This shows that $\norm[\LL^{\infty}(\Omega)]{\uu_h^{i,\ell+1}} \le 1$
and concludes the proof of part~{\rm(i)}.

Let $\uu_h^{i,\ell+1}$ and $\uu_h^{i,\ell+2}$
(resp., $\zz_h^{i,\ell+1}$ and $\zz_h^{i,\ell+2}$)
be two consecutive iterates
satisfying~\eqref{eq:mpsx1}
(resp., \eqref{eq:mpsx2}).
Taking the difference of the equations satisfied by
$\zz_h^{i,\ell+2}$ and $\zz_h^{i,\ell+1}$
and choosing $\ppsi_h = \zz_h^{i,\ell+2} - \zz_h^{i,\ell+1}$,
we obtain the identity
\begin{equation} \label{eq:aux1}
\begin{split}
& 2 \tau \norm[h]{\zz_h^{i,\ell+2}-\zz_h^{i,\ell+1}}^2 \\
& \quad = k \inner[h]{\uu_h^{i,\ell+2} \times \Ph\heff[\uu_h^{i,\ell+2}] - \uu_h^{i,\ell+1} \times \Ph\heff[\uu_h^{i,\ell+1}]}{\zz_h^{i,\ell+2}-\zz_h^{i,\ell+1}} \\
& \qquad + 2 \alpha \inner[h]{(\uu_h^{i,\ell+2} - \uu_h^{i,\ell+1}) \times \mm_h^i}{\zz_h^{i,\ell+2}-\zz_h^{i,\ell+1}} \\
& \qquad - k \inner[h]{\uu_h^{i,\ell+2} \times \zz_h^{i,\ell+2} - \uu_h^{i,\ell+1} \times \zz_h^{i,\ell+1}}{\zz_h^{i,\ell+2}-\zz_h^{i,\ell+1}}.
\end{split}
\end{equation}
Taking the difference of the equations satisfied by
$\uu_h^{i,\ell+2}$ and $\uu_h^{i,\ell+1}$
and choosing the test functions
$\pphi_h = \uu_h^{i,\ell+2} - \uu_h^{i,\ell+1}$
and $\pphi_h = \zz_h^{i,\ell+2} - \zz_h^{i,\ell+1}$,
we obtain the identities
\begin{gather}
 \label{eq:aux2}
2\norm[h]{\uu_h^{i,\ell+2}-\uu_h^{i,\ell+1}}^2
= - k \inner[h]{\uu_h^{i,\ell+1} \times (\zz_h^{i,\ell+1} - \zz_h^{i,\ell})}{\uu_h^{i,\ell+2}-\uu_h^{i,\ell+1}},\\
\label{eq:aux3}
2\inner[h]{\uu_h^{i,\ell+2}-\uu_h^{i,\ell+1}}{\zz_h^{i,\ell+2} - \zz_h^{i,\ell+1}}
+ k \inner[h]{\uu_h^{i,\ell+2} \times \zz_h^{i,\ell+1} - \uu_h^{i,\ell+1} \times \zz_h^{i,\ell}}{\zz_h^{i,\ell+2} - \zz_h^{i,\ell+1}}
= 0.
\end{gather}
From~\eqref{eq:aux2}, since $\norm[\LL^{\infty}(\Omega)]{\uu_h^{i,\ell+1}} \le 1$ from part~{\rm(i)}, we deduce that
\begin{equation} \label{eq:aux4}
\norm[h]{\uu_h^{i,\ell+2}-\uu_h^{i,\ell+1}}
\leq \frac{k}{2} \norm[h]{\zz_h^{i,\ell+1} - \zz_h^{i,\ell}}.
\end{equation}
Combining~\eqref{eq:aux1} and~\eqref{eq:aux3}, we obtain that
\begin{equation*}
\begin{split}
& 2 \tau \norm[h]{\zz_h^{i,\ell+2}-\zz_h^{i,\ell+1}}^2 \\
& \quad
\stackrel{\eqref{eq:aux1}}{=}
k \inner[h]{(\uu_h^{i,\ell+2} - \uu_h^{i,\ell+1}) \times \Ph\heff[\uu_h^{i,\ell+2}]}{\zz_h^{i,\ell+2}-\zz_h^{i,\ell+1}} \\
& \qquad +
k \inner[h]{\uu_h^{i,\ell+1} \times (\Ph\heff[\uu_h^{i,\ell+2}]-\Ph\heff[\uu_h^{i,\ell+1}])}{\zz_h^{i,\ell+2}-\zz_h^{i,\ell+1}} \\
& \qquad
+ 2 \alpha \inner[h]{(\uu_h^{i,\ell+2} - \uu_h^{i,\ell+1}) \times \mm_h^i}{\zz_h^{i,\ell+2}-\zz_h^{i,\ell+1}} \\
& \qquad
- k \inner[h]{\uu_h^{i,\ell+2} \times \zz_h^{i,\ell+2} - \uu_h^{i,\ell+1} \times \zz_h^{i,\ell+1}}{\zz_h^{i,\ell+2}-\zz_h^{i,\ell+1}} \\
& \quad
\stackrel{\eqref{eq:aux3}}{=}
k \inner[h]{(\uu_h^{i,\ell+2} - \uu_h^{i,\ell+1}) \times \Ph\heff[\uu_h^{i,\ell+2}]}{\zz_h^{i,\ell+2}-\zz_h^{i,\ell+1}} \\
& \qquad + k \inner[h]{\uu_h^{i,\ell+1} \times (\Ph\heff[\uu_h^{i,\ell+2}]-\Ph\heff[\uu_h^{i,\ell+1}])}{\zz_h^{i,\ell+2}-\zz_h^{i,\ell+1}} \\
& \qquad
+ 2 \alpha \inner[h]{(\uu_h^{i,\ell+2} - \uu_h^{i,\ell+1}) \times \mm_h^i}{\zz_h^{i,\ell+2}-\zz_h^{i,\ell+1}} \\
& \qquad
- k \inner[h]{\uu_h^{i,\ell+2} \times \zz_h^{i,\ell+2} - \uu_h^{i,\ell+1} \times \zz_h^{i,\ell+1}}{\zz_h^{i,\ell+2}-\zz_h^{i,\ell+1}} \\
& \qquad
+ 2\inner[h]{\uu_h^{i,\ell+2}-\uu_h^{i,\ell+1}}{\zz_h^{i,\ell+2} - \zz_h^{i,\ell+1}}
+ k \inner[h]{\uu_h^{i,\ell+2} \times \zz_h^{i,\ell+1} - \uu_h^{i,\ell+1} \times \zz_h^{i,\ell}}{\zz_h^{i,\ell+2} - \zz_h^{i,\ell+1}} \\
& \quad
\stackrel{\phantom{\eqref{eq:aux3}}}{=}
k \inner[h]{(\uu_h^{i,\ell+2} - \uu_h^{i,\ell+1}) \times \Ph\heff[\uu_h^{i,\ell+2}]}{\zz_h^{i,\ell+2}-\zz_h^{i,\ell+1}} \\
& \qquad + k \inner[h]{\uu_h^{i,\ell+1} \times (\Ph\heff[\uu_h^{i,\ell+2}]-\Ph\heff[\uu_h^{i,\ell+1}])}{\zz_h^{i,\ell+2}-\zz_h^{i,\ell+1}} \\
& \qquad
+ 2 \alpha \inner[h]{(\uu_h^{i,\ell+2} - \uu_h^{i,\ell+1}) \times \mm_h^i}{\zz_h^{i,\ell+2}-\zz_h^{i,\ell+1}} \\
& \qquad
+ 2\inner[h]{\uu_h^{i,\ell+2}-\uu_h^{i,\ell+1}}{\zz_h^{i,\ell+2} - \zz_h^{i,\ell+1}}
+ k \inner[h]{\uu_h^{i,\ell+1} \times (\zz_h^{i,\ell+1} - \zz_h^{i,\ell})}{\zz_h^{i,\ell+2} - \zz_h^{i,\ell+1}}.
\end{split}
\end{equation*}
It follows that
\begin{equation*}
\begin{split}
& 2 \tau \norm[h]{\zz_h^{i,\ell+2}-\zz_h^{i,\ell+1}} \\
& \quad \le
\big(2+ 2 \alpha \norm[\LL^{\infty}(\Omega)]{\mm_h^i} + k \norm[\LL^{\infty}(\Omega)]{\Ph\heff[\uu_h^{i,\ell+2}]}\big)
\norm[h]{\uu_h^{i,\ell+2} - \uu_h^{i,\ell+1}} \\
& \qquad +
k \norm[\LL^{\infty}(\Omega)]{\uu_h^{i,\ell+1}}
\norm[h]{\Ph\heff[\uu_h^{i,\ell+2}]-\Ph\heff[\uu_h^{i,\ell+1}]}
+ k \norm[\LL^{\infty}(\Omega)]{\uu_h^{i,\ell+1}}
\norm[h]{\zz_h^{i,\ell+1} - \zz_h^{i,\ell}}.
\end{split}
\end{equation*}
Using that $\norm[\LL^{\infty}]{\mm_h^i} = 1$
and $\norm[\LL^{\infty}]{\uu_h^{i,\ell+1}} \le 1$,
together with the estimates~\cite{bp2006,bartels2015a}
\begin{align*}
\norm[\LL^{\infty}]{\Ph\heff[\uu_h^{i,\ell+2}]} &\le \Cinv^2 \hmin^{-2}\norm[\LL^{\infty}]{\uu_h^{i,\ell+2}},\\
\norm[h]{\Ph\heff[\uu_h^{i,\ell+2}]} &\le \Cinv^2 \hmin^{-2} \norm[h]{\uu_h^{i,\ell+2}}
\end{align*}
(where $\Cinv>0$ depends only on the shape-regularity of $\T_h$)
and~\eqref{eq:aux4},
we obtain~\eqref{eq:contraction} with
$q = [(2 + \alpha + \tau)k + \Cinv^2 k^2 \hmin^{-2}] / (2 \tau)$.
Hence,
if $k < k_0 := \tau/(2 + \alpha + \tau)$
and $k < C \hmin$ (with $C = \sqrt{\tau}/\Cinv$),
then $0 < q < 1$.
This proves part~{\rm(ii)}.

By construction,
$\mm_h^{i+1} = 2 \uu_h^{i,\ell_i+1} - \mm_h^i$ satisfies
\begin{equation*}
\inner[h]{d_t\mm_h^{i+1}}{\pphi_h}
=  \inner[h]{\mm_h^{i+1/2} \times \zz_h^{i,\ell}}{\pphi_h}
\quad
\text{for all } \pphi_h \in \S^1(\T_h)^3,
\end{equation*}
while $\ww_h^{i+1} = 2 \zz_h^{i,\ell_i+1} - \ww_h^i$ satisfies~\eqref{eq:mps2}.
The argument used to prove Proposition~\ref{prop:conservation}
shows that $\mm_h^{i+1} \in \Mh$ and $\ww_h^{i+1} \in \Kh[\mm_h^{i+1}]$.
This shows part~{\rm(iii)} and concludes the proof.
\end{proof}
%%%%%%%%%%%%%%%%%%%%

%%%%%%%%%%%%%%%%%%%%
\subsection{Discrete energy laws and stability} \label{sec:stability}
%%%%%%%%%%%%%%%%%%%%

We prove the discrete energy laws satisfied by the approximations generated by the algorithms.

%%%%%%%%%%%%%%%%%%%%
\begin{proof}[Proof of Proposition~\ref{prop:energy}]
Let $i \in \N_0$.
We test~\eqref{eq:tps1} with $\pphi_h=\vv_h^{i+1} \in \Kh[\mm_h^i]$
and multiply the resulting equation by $k$.
We obtain the identity
\begin{equation*}
\tau \inner[h]{\vv_h^{i+1}-\vv_h^i}{\vv_h^{i+1}}
+ \alpha k \norm[h]{\vv_h^{i+1}}^2
+ k^2 \norm[\LL^2(\Omega)]{\Grad\vv_h^{i+1}}^2 \\
= - k \inner{\Grad\mm_h^i}{\Grad\vv_h^{i+1}}
\end{equation*}
Since the angle condition~\eqref{eq:angleCondition} is satisfied,
\cite[Lemma~3.2]{bartels2005} yields that
\begin{equation} \label{eq:nodalProjectionEnergy}
\norm[\LL^2(\Omega)]{\Grad\mm_h^{i+1}}
\le
\norm[\LL^2(\Omega)]{\Grad\mm_h^i + k \Grad\vv_h^{i+1}}.
\end{equation}
Hence, we obtain that
\begin{equation*}
\begin{split}
\E[\mm_h^{i+1}]
\stackrel{\eqref{eq:nodalProjectionEnergy}}{\leq}
\frac{1}{2} \norm[\LL^2(\Omega)]{\Grad\mm_h^i + k \Grad\vv_h^{i+1}}^2
= \E[\mm_h^i]
+ k \inner{\Grad\mm_h^i}{\Grad\vv_h^{i+1}}
+ \frac{k^2}{2} \norm[\LL^2(\Omega)]{\Grad\vv_h^{i+1}}^2.
\end{split}
\end{equation*}
Altogether, we obtain that
\begin{equation*}
\E[\mm_h^{i+1}]
+ \tau \inner[h]{\vv_h^{i+1}-\vv_h^i}{\vv_h^{i+1}}
+ \alpha k \norm[h]{\vv_h^{i+1}}^2
+ \frac{k^2}{2} \norm[\LL^2(\Omega)]{\Grad\vv_h^{i+1}}^2
\leq \E[\mm_h^i].
\end{equation*}
Applying the vector identity
\begin{equation*}
(\vec{a}-\vec{b})\cdot\vec{a}
= \frac{1}{2}\abs{\vec{a}}^2
- \frac{1}{2}\abs{\vec{b}}^2
+ \frac{1}{2}\abs{\vec{a} - \vec{b}}^2
\quad
\text{for all } \vec{a},\vec{b} \in\R^3
\end{equation*}
to the second term on the left-hand side yields~\eqref{eq:tps:energy}.
This proves part~{\rm(i)}.

To show the stability of Algorithm~\ref{alg:mps},
we choose
$\pphi_h = d_t \mm_h^{i+1}$ in~\eqref{eq:mps1},
$\pphi_h = \Ph\heff[\mm_h^{i+1/2}]$ in~\eqref{eq:mps1},
and
$\ppsi_h = \ww_h^{i+1/2}$ in~\eqref{eq:mps2}
to obtain the identities
\begin{gather*}
\norm[h]{d_t\mm_h^{i+1}}^2
=
- \inner[h]{\mm_h^{i+1/2} \times \ww_h^{i+1/2}}{d_t\mm_h^{i+1}}, \\
\inner[h]{d_t\mm_h^{i+1}}{\Ph\heff[\mm_h^{i+1/2}]}
=
- \inner[h]{\mm_h^{i+1/2} \times \ww_h^{i+1/2}}{\Ph\heff[\mm_h^{i+1/2}]}, \\
 \tau \inner[h]{d_t\ww_h^{i+1}}{\ww_h^{i+1/2}}
=
\inner[h]{\mm_h^{i+1/2} \times \Ph\heff[\mm_h^{i+1/2}]}{\ww_h^{i+1/2}}
- \alpha \inner[h]{\mm_h^{i+1/2} \times d_t\mm_h^{i+1}}{\ww_h^{i+1/2}},
\end{gather*}
respectively.
Combining these three equations, we obtain that
\begin{equation*}
\inner[h]{d_t\mm_h^{i+1}}{\Ph\heff[\mm_h^{i+1/2}]}
=
 \tau \inner[h]{d_t\ww_h^{i+1}}{\ww_h^{i+1/2}}
+
\alpha \norm[h]{d_t\mm_h^{i+1}}^2. 
\end{equation*}
Since
\begin{gather*}
\inner[h]{\Ph \heff[\mm_h^{i + 1/2}]}{d_t \mm_h^{i+1}}
\stackrel{\eqref{eq:pseudo-projection},\eqref{eq:heff},\eqref{eq:llg:energy}}{=}
- \frac{1}{k} \big( \E[\mm_h^{i + 1}] - \E[\mm_h^i] \big), \\
\inner[h]{d_t\ww_h^{i+1}}{\ww_h^{i+1/2}}
=
\frac{1}{2k} \big( \norm[h]{\ww_h^{i+1}}^2 - \norm[h]{\ww_h^i}^2 \big),
\end{gather*}
we obtain~\eqref{eq:mps:energy}.
This proves part~{\rm(ii)}.
The proof of~\eqref{eq:mpsx:energy} from part~{\rm(iii)} can be obtained with the very same argument.
\end{proof}
%%%%%%%%%%%%%%%%%%%%

%%%%%%%%%%%%%%%%%%%%
\subsection{Convergence results} \label{sec:convergence}
%%%%%%%%%%%%%%%%%%%%

First, we prove the convergence of Algorithm~\ref{alg:tps}.

%%%%%%%%%%%%%%%%%%%%%
\begin{proof}[Proof of Theorem~\ref{thm:convergence} for Algorithm~\ref{alg:tps}]
The proof is largely based on the argument of~\cite{alouges2008a,bffgpprs2014,hpprss2019}.
We start with recalling the estimates
\begin{equation} \label{eq:geometric}
\lvert \mm_h^{i+1}(z) - \mm_h^i(z) \rvert
\le k \lvert \vv_h^{i+1}(z) \rvert
\quad
\text{and}
\quad
\lvert \mm_h^{i+1}(z) - \mm_h^i(z) - k \vv_h^{i+1}(z)\rvert
\le \frac{k^2}{2} \lvert \vv_h^{i+1}(z) \rvert^2,
\end{equation}
which hold for all $i \in \N_0$ and $z \in \Nh$; see~\cite{aj2006,bkp2008}.
The first inequality in~\eqref{eq:geometric}, together with~\eqref{eq:normEquivalence},
yields that
$\norm[h]{d_t\mm_h^{i+1}} \le \norm[h]{\vv_h^{i+1}}$
for all $i \in \N_0$.

Let $T>0$ be arbitrary.
With the uniform boundedness guaranteed by~\eqref{eq:tps:energy} and~\eqref{eq:convergence0},
we can construct
$\mm\in L^{\infty}(0,\infty;\HH^1(\Omega)) \cap W^{1,\infty}(0,\infty;\LL^2(\Omega))$,
satisfying $\abs{\mm}=1$ a.e.\ in $\Omega \times (0,\infty)$,
such that,
upon extraction of (nonrelabeled) subsequences,
there hold the convergences
$\mm_{hk}, \mm_{hk}^{\pm} \weakstarto \mm$ in $L^{\infty}(0,\infty;\HH^1(\Omega))$,
$\mm_{hk}\vert_{\Omega_T} \weakto \mm\vert_{\Omega_T}$ in $\HH^1(\Omega_T)$,
$\partial_t\mm_{hk}, \vv_{hk}^{+} \weakstarto \partial_t\mm$ in $L^{\infty}(0,\infty;\LL^2(\Omega))$,
as well as $\mm_{hk}, \mm_{hk}^{\pm} \to \mm$ in $\LL^2(\Omega \times (0,\infty))$
and pointwise almost everywhere in $\Omega \times (0,\infty)$.
Moreover, we have that
$k \Grad\vv_{hk}^+ \to 0$
in $L^2(0,\infty;\LL^2(\Omega))$
as $h,k \to 0$.

Let $\vvphi\in C^{\infty}_c([0,T);\CC(\overline{\Omega}))$ be an arbitrary smooth test function.
Let $N \in \N$ be the smallest integer such that $T \le kN = t_N$.
Let $i \in \{ 0, \dots, N-1\}$.
We choose the test function
$\pphi_h=\Interp[\mm_h^i \times \vvphi(t_i)]\in\Kh[\mm_h^i]$
in~\eqref{eq:tps1}
to obtain
\begin{equation*}
\begin{split}
& \tau \inner[h]{d_t\vv_h^{i+1}}{\Interp[\mm_h^i \times \vvphi(t_i)]}
+ \alpha \inner[h]{\vv_h^{i+1}}{\Interp[\mm_h^i \times \vvphi(t_i)]}
+ \inner[h]{\mm_h^i\times\vv_h^{i+1}}{\Interp[\mm_h^i \times \vvphi(t_i)]} \\
& \quad - k \inner[h]{\Ph\heff[\vv_h^{i+1}]}{\Interp[\mm_h^i \times \vvphi(t_i)]}
= \inner[h]{\Ph\heff[\mm_h^i]}{\Interp[\mm_h^i \times \vvphi(t_i)]},
\end{split}
\end{equation*}
where we extend $\vvphi$ by zero in $(T,t_N)$.
Due to the presence of the mass-lumped scalar product,
we can remove the nodal interpolant from the first three terms on the left-hand side
without affecting the value of the integrals.
Then,
multiplying the latter by $k$,
summing over $i = 0, \dots, N-1$,
and using~\eqref{eq:heff} and~\eqref{eq:pseudo-projection}, we obtain the identity
\begin{equation*}
\begin{split}
& \tau k \sum_{i=0}^{N-1} \inner[h]{d_t\vv_h^{i+1}}{\mm_h^i \times \vvphi(t_i)}
+ \alpha k \sum_{i=0}^{N-1} \inner[h]{\vv_h^{i+1}}{\mm_h^i \times \vvphi(t_i)}
+ k \sum_{i=0}^{N-1} \inner[h]{\mm_h^i\times\vv_h^{i+1}}{\mm_h^i \times \vvphi(t_i)} \\
& \quad 
= - k \sum_{i=0}^{N-1} \inner{\Grad(\mm_h^i + k \vv_h^{i+1})}{\Grad\Interp[\mm_h^i \times \vvphi(t_i)]}.
\end{split}
\end{equation*}
Next, we rewrite the first term on the left-hand side
using the summation by parts formula
\begin{equation} \label{eq:summation}
\sum_{i=0}^{N-1} (a_{i+1}-a_i) b_i
= - \sum_{i=0}^{N-1} a_{i+1} (b_{i+1}-b_i) + a_N b_N - a_0 b_0
\quad
\text{for all sequences } \{a_i\}_{i=0}^{N}, \{b_i\}_{i=0}^{N}
\end{equation}
and performing some algebraic manipulations:
\begin{equation*}
\begin{split}
& \tau k \sum_{i=0}^{N-1} \inner[h]{d_t\vv_h^{i+1}}{\mm_h^i \times \vvphi(t_i)}
\\
& \quad =
- \tau \sum_{i=0}^{N-1} \inner[h]{\vv_h^{i+1}}{\mm_h^{i+1} \times \vvphi(t_{i+1}) - \mm_h^i \times \vvphi(t_i)} \\
& \qquad
+ \tau \inner[h]{\vv_h^N}{\mm_h^N \times \vvphi(t_N)}
- \tau \inner[h]{\vv_h^0}{\mm_h^0 \times \vvphi(0)}
\\
& \quad =
- \tau k \sum_{i=0}^{N-1} \inner[h]{\vv_h^{i+1}}{d_t\mm_h^{i+1} \times \vvphi(t_{i+1})}
- \tau k \sum_{i=0}^{N-1} \inner[h]{\vv_h^{i+1}}{\mm_h^i \times d_t\vvphi(t_{i+1})} \\
& \qquad
+ \tau \inner[h]{\vv_h^N}{\mm_h^N \times \vvphi(t_N)}
- \tau \inner[h]{\vv_h^0}{\mm_h^0 \times \vvphi(0)}
\\
& \quad =
- \tau k \sum_{i=0}^{N-1} \inner[h]{\vv_h^{i+1}}{(d_t\mm_h^{i+1} - \vv_h^{i+1}) \times \vvphi(t_{i+1})}
+ \tau k \sum_{i=0}^{N-1} \inner[h]{\mm_h^i \times \vv_h^{i+1}}{d_t\vvphi(t_{i+1})} \\
& \qquad
- \tau \inner[h]{\mm_h^N \times \vv_h^N}{\vvphi(t_N)}
+ \tau \inner[h]{\mm_h^0 \times \vv_h^0}{\vvphi(0)}.
\end{split}
\end{equation*}
Using H\"older inequality,
a combination of the second inequality in~\eqref{eq:geometric}
and the norm equivalence~\eqref{eq:normEquivalence},
and inverse estimates
(see, e.g., \cite[Lemma~3.5]{bartels2015}),
the first term on the right-hand side can be estimated as
\begin{equation*}
\begin{split}
& \tau k \sum_{i=0}^{N-1} \inner[h]{\vv_h^{i+1}}{(d_t\mm_h^{i+1} - \vv_h^{i+1}) \times \vvphi(t_{i+1})} \\
& \quad \lesssim
k \sum_{i=0}^{N-1} \norm[\LL^3(\Omega)]{\vv_h^{i+1}}
\norm[\LL^{3/2}(\Omega)]{d_t \mm_h^{i+1} - \vv_h^{i+1}}
\norm[\LL^{\infty}(\Omega)]{\vvphi(t_{i+1})} \\
& \quad \lesssim k^2
\sum_{i=0}^{N-1} \norm[\LL^{3}(\Omega)]{\vv_h^{i+1}}^3
\lesssim k^2 \hmin^{-d/2}
\sum_{i=0}^{N-1} \norm[\LL^{2}(\Omega)]{\vv_h^{i+1}}^3
\lesssim k \hmin^{-d/2}.
\end{split}
\end{equation*}
Altogether,
using also the identity
\begin{equation*}
[\mm_h^i(z)\times\vv_h^{i+1}(z)]\cdot[\mm_h^i(z) \times \vvphi(z,t_i)] = \vv_h^{i+1}(z)\cdot\vvphi(z,t_i)
\quad
\text{for all } z \in \Nh
\end{equation*}
(which follows from~\eqref{eq:triple}, $\mm_h^i \in \Mh$ and $\vv_h^{i+1} \in \Kh[\mm_h^i]$)
and
observing that $\vvphi(t_N) = \0$,
we thus obtain that
\begin{equation*}
\begin{split}
& k \sum_{i=0}^{N-1} \inner[h]{\vv_h^{i+1}}{\vvphi(t_i)} \\
& \quad 
= - k \sum_{i=0}^{N-1} \inner{\Grad(\mm_h^i + k \vv_h^{i+1})}{\Grad\Interp[\mm_h^i \times \vvphi(t_i)]}
+ \alpha k \sum_{i=0}^{N-1} \inner[h]{\mm_h^i \times \vv_h^{i+1}}{\vvphi(t_i)}
\\
& \qquad - \tau k \sum_{i=0}^{N-1} \inner[h]{\mm_h^i \times \vv_h^{i+1}}{d_t\vvphi(t_{i+1})}
- \tau \inner[h]{\mm_h^0 \times \vv_h^0}{\vvphi(0)}
+ o(1).
\end{split}
\end{equation*}
Using the approximation properties of the nodal interpolant
and estimate~\eqref{eq:h-scalar-product} (see the argument of~\cite{bp2006,alouges2008a}),
we can replace all mass-lumped scalar products by $L^2$-products
and remove the nodal interpolant from the first term on the left-hand side
at the price of an error which goes to zero in the limit.
Rewriting the space-time integrals of
the resulting equation in terms of the time reconstructions~\eqref{eq:reconstructions},
we obtain that
\begin{equation*}
\begin{split}
& \int_0^T \inner{\vv_{hk}^+(t)}{\vvphi_k^-(t)} \dt \\
& \quad 
= - \int_0^T \inner{\Grad(\mm_{hk}^-(t) + k \vv_{hk}^+(t))}{\Grad(\mm_{hk}^-(t) \times \vvphi_k^-(t))} \dt
+ \alpha \int_0^T \inner{\mm_{hk}^-(t) \times \vv_{hk}^+(t)}{\vvphi_k^-(t)} \dt
\\
& \qquad - \tau \int_0^T \inner{\mm_{hk}^-(t) \times \vv_{hk}^+(t)}{\partial_t\vvphi_k(t)} \dt
- \tau \inner{\mm_h^0 \times \vv_h^0}{\vvphi(0)}
+ o(1).
\end{split}
\end{equation*}
Using the available convergence results,
we can pass to the limit as $h,k \to 0$ the latter and obtain that each term converges towards
the corresponding one in~\eqref{eq:weak:variational}.
By density, it follows that $\mm$ satisfies~\eqref{eq:weak:variational}
for all $\vvphi\in C^{\infty}_c([0,T);\HH^1(\Omega))$.
This shows that $\mm$ satisfies part~{\rm(iii)} of Definition~\ref{def:weak}.

The proof that $\mm$ attains the prescribed initial data $(\mm^0,\vv^0)$ continuously in
$\HH^1(\Omega) \times \LL^2(\Omega)$ (part~{\rm(ii)} of Definition~\ref{def:weak})
follows the argument of~\cite[page~72]{bfp2007}.
The energy inequality~\eqref{eq:weak:energy}
(part~{\rm(iv)} of Definition~\ref{def:weak})
can be obtained from the discrete energy law~\eqref{eq:tps:energy}
using the available convergence results and standard lower semicontinuity arguments.
\end{proof}
%%%%%%%%%%%%%%%%%%%%

Next, we prove the convergence result for the nonlinear angular momentum method.

%%%%%%%%%%%%%%%%%%%%
\begin{proof}[Proof of Theorem~\ref{thm:convergence} for Algorithm~\ref{alg:mps}]
The proof follows the ideas of~\cite{bp2006,kw2014,bartels2015a}.
Testing~\eqref{eq:mps1} with $\pphi_h = d_t\mm_h^{i+1}$,
we infer that $\norm[h]{d_t\mm_h^{i+1}} \le \norm[h]{\ww_h^{i+1/2}}$ for all $i \in \N_0$.

Let $T>0$ be arbitrary.
With the uniform boundedness guaranteed by~\eqref{eq:mps:energy} and~\eqref{eq:convergence0},
we can construct
$\mm\in L^{\infty}(0,\infty;\HH^1(\Omega)) \cap W^{1,\infty}(0,\infty;\LL^2(\Omega))$
and 
$\ww\in L^{\infty}(0,\infty;\LL^2(\Omega))$
such that,
upon extraction of a (nonrelabeled) subsequence,
we have the convergences
$\mm_{hk}, \mm_{hk}^+, \overline{\mm}_{hk} \weakstarto \mm$ in $L^{\infty}(0,\infty;\HH^1(\Omega))$,
$\partial_t\mm_{hk} \weakstarto \partial_t\mm$ in $L^{\infty}(0,\infty;\LL^2(\Omega))$,
$\mm_{hk}\vert_{\Omega_T} \weakto \mm\vert_{\Omega_T}$ in $\HH^1(\Omega_T)$,
$\mm_{hk}, \mm_{hk}^+, \overline{\mm}_{hk} \to \mm$ in $\LL^2(\Omega \times (0,\infty))$
and pointwise almost everywhere in $\Omega \times (0,\infty)$,
as well as
$\overline{\ww}_{hk}, \ww_{hk}^+ \weakstarto \ww$ in $L^{\infty}(0,\infty;\LL^2(\Omega))$.
Moreover, it holds that $\abs{\mm}=1$
and $\mm \cdot \ww = 0$
a.e.\ in $\Omega \times (0,\infty)$.

Let $\zzeta, \vvphi \in C^{\infty}_c([0,T);\CC(\overline{\Omega}))$ be arbitrary smooth test functions.
Let $N \in \N$ be the smallest integer such that $T \le kN = t_N$.
Let $i \in \{ 0, \dots, N-1\}$.
We choose the test function
$\pphi_h=\Interp[\zzeta(t_i)]\in\S^1(\T_h)^3$ in~\eqref{eq:mps1}
and $\ppsi_h=\Interp[\vvphi(t_i)]\in\S^1(\T_h)^3$ in~\eqref{eq:mps2}.
Multiplying the resulting equation by $k$,
summing over $i = 0, \dots, N-1$,
and using~\eqref{eq:heff} and~\eqref{eq:pseudo-projection} on the term which involves the effective field,
we obtain the identities
\begin{align*}
k \sum_{i=0}^{N-1}\inner[h]{d_t\mm_h^{i+1}}{\zzeta(t_i)}
& =
- k \sum_{i=0}^{N-1} \inner[h]{\mm_h^{i+1/2} \times \ww_h^{i+1/2}}{\zzeta(t_i)}, \\
 \tau k \sum_{i=0}^{N-1} \inner[h]{d_t\ww_h^{i+1}}{\vvphi(t_i)}
& = - k \sum_{i=0}^{N-1}
\inner{\Grad\mm_h^{i+1/2}}{\Grad(\vvphi(t_i) \times \mm_h^{i+1/2})} \\
- \alpha k \sum_{i=0}^{N-1} & \inner[h]{\mm_h^{i+1/2} \times d_t \mm_h^{i+1}}{\vvphi(t_i)}
- k \sum_{i=0}^{N-1} \inner[h]{\mm_h^{i+1/2} \times \ww_h^{i+1/2}}{\vvphi(t_i)},
\end{align*}
where we extend $\zzeta , \vvphi$ by zero in $(T,t_N)$.
Next, we rewrite the first term on the left-hand side of both equations
using the summation by parts formula~\eqref{eq:summation}:
\begin{align*}
- k \sum_{i=0}^{N-1}\inner[h]{\mm_h^{i+1}}{d_t \zzeta(t_{i+1})}
- \inner[h]{\mm_h^0}{\zzeta(0)}
& =
- k \sum_{i=0}^{N-1} \inner[h]{\mm_h^{i+1/2} \times \ww_h^{i+1/2}}{\zzeta(t_i)}, \\
-  \tau k \sum_{i=0}^{N-1} \inner[h]{\ww_h^{i+1}}{d_t\vvphi(t_{i+1})}
-  \tau \inner[h]{\ww_h^0}{\vvphi(0)}
& = - k \sum_{i=0}^{N-1}
\inner{\Grad\mm_h^{i+1/2}}{\Grad(\vvphi(t_i) \times \mm_h^{i+1/2})} \\
- \alpha k \sum_{i=0}^{N-1} \inner[h]{\mm_h^{i+1/2} \times d_t & \mm_h^{i+1}}{\vvphi(t_i)}
 - k \sum_{i=0}^{N-1} \inner[h]{\mm_h^{i+1/2} \times \ww_h^{i+1/2}}{\vvphi(t_i)}.
\end{align*}
Rewriting the space-time integrals of both equations
in terms of the time reconstructions defined in~\eqref{eq:reconstructions},
we obtain that
\begin{align*}
- \int_0^T \inner[h]{\mm_{hk}^+(t)}{\partial_t \zzeta_k(t)} \dt
- \inner[h]{\mm_h^0}{\zzeta(0)}
& =
- \int_0^T \inner[h]{\overline{\mm}_{hk}(t) \times \overline{\ww}_{hk}(t)}{\zzeta_k^-(t)} \dt, \\
-  \tau \int_0^T \inner[h]{\ww_{hk}^+(t)}{\partial_t\vvphi_k(t)} \dt
-  \tau \inner[h]{\ww_h^0}{\vvphi(0)}
& = - \int_0^T \inner{\Grad\overline{\mm}_{hk}(t)}{\Grad(\vvphi_k^-(t) \times \overline{\mm}_{hk}(t))} \dt \\
- \alpha \int_0^T \inner[h]{\overline{\mm}_{hk}(t) \times \partial_t \mm_{hk}(t)&}{\vvphi_k^-(t)} \dt
- \int_0^T \inner[h]{\overline{\mm}_{hk}(t) \times \overline{\ww}_{hk}(t)}{\vvphi_k^-(t)} \dt.
\end{align*}
Using the available convergence results and~\eqref{eq:h-scalar-product},
we can proceed as in~\cite[Section~3]{bp2006}
and pass the latter equations to the limit as $h,k \to 0$.
Rearranging the terms, we obtain that
\begin{align}
\label{eq:prefinal_var1}
- \int_0^T \inner{\mm(t)}{\partial_t \zzeta(t)} \dt
& =
- \int_0^T \inner{\mm(t) \times \ww(t)}{\zzeta(t)} \dt
+ \inner{\mm^0}{\zzeta(0)}, \\
- \int_0^T \inner{\mm(t) \times \ww(t)}{\vvphi(t)} \dt
& =
\notag
- \int_0^T \inner{\Grad\mm(t)}{\Grad(\vvphi(t) \times \mm(t))} \dt \\
\label{eq:prefinal_var2}
+ \alpha \int_0^T \inner{\mm(t) \times \partial_t \mm(t)&}{\vvphi(t)} \dt
-  \tau \int_0^T \inner{\ww(t)}{\partial_t\vvphi(t)} \dt
-  \tau \inner{\mm^0 \times \vv^0}{\vvphi(0)}.
\end{align}
To conclude,
following~\cite[Section~3.3]{kw2014},
we observe that~\eqref{eq:prefinal_var1} reveals that
\begin{equation} \label{eq:final1}
\mmt = - \mm \times \ww \quad \text{a.e.\ in } \Omega \times (0,T).
\end{equation}
Since $\abs{\mm} = 1$ and $\mm\cdot\ww = 0$,
using~\eqref{eq:triple}, it follows that
\begin{equation} \label{eq:final2}
\ww = \mm \times \mmt \quad \text{a.e.\ in } \Omega \times (0,T).
\end{equation}
Using~\eqref{eq:final1} and~\eqref{eq:final2} in the term on the left-hand side
and in the third term on the right-hand side of~\eqref{eq:prefinal_var2},
respectively,
we obtain the variational formulation~\eqref{eq:weak:variational}.
By density, it follows that $\mm$ satisfies~\eqref{eq:weak:variational}
for all $\vvphi\in C^{\infty}_c([0,T);\HH^1(\Omega))$.
This shows that $\mm$ satisfies part~{\rm(iii)} of Definition~\ref{def:weak}.
The verification of part~{\rm(ii)}
and~{\rm(iv)} can be performed
using the argument employed for Algorithm~\ref{alg:mps}.
This proves the desired convergence and concludes the proof.
\end{proof}
%%%%%%%%%%%%%%%%%%%%

In the following proposition,
we establish the boundedness of the approximations generated by Algorithm~\ref{alg:mpsx}.

%%%%%%%%%%%%%%%%%%%%
\begin{proposition} \label{prop:boundedness}
Let $j \in \N$.
Assume that $\eps = \OO(\hmin)$ as $h,\eps \to 0$.
Then, there exist thresholds $h_0, k_0, \eps_0 >0$ such that,
if $h < h_0$, $k < k_0$, and $\eps < \eps_0$,
the approximations generated by Algorithm~\ref{alg:mpsx} satisfy the inequality
\begin{equation} \label{eq:stability_mpsx}
\norm[\LL^2(\Omega)]{\Grad\mm_h^j}^2
+ \norm[h]{\ww_h^j}^2
+ k \sum_{i=0}^{j-1} \norm[h]{d_t\mm_h^{i+1}}^2
\le C (1 + jk) \exp(jk).
\end{equation}
The thresholds $h_0, k_0, \eps_0 >0$ and the constant $C>0$
depend only on the shape-regularity of $\T_h$ and the problem data.
\end{proposition}
%%%%%%%%%%%%%%%%%%%%

%%%%%%%%%%%%%%%%%%%%
\begin{proof}
Let $i = 0, \dots, j-1$.
Proposition~\eqref{prop:energy}{\rm(iii)} yields~\eqref{eq:mpsx:energy}.
Taking the sum over $i=0,\dots,j-1$, we obtain the identity
\begin{equation*}
\begin{split}
& \frac{1}{2} \norm[\LL^2(\Omega)]{\Grad\mm_h^j}^2
+ \frac{\tau}{2} \norm[h]{\ww_h^j}^2
+ \alpha k \sum_{i=0}^{j-1} \norm[h]{d_t\mm_h^{i+1}}^2 \\
& \quad = \frac{1}{2} \norm[\LL^2(\Omega)]{\Grad\mm_h^0}^2
+ \frac{\tau}{2} \norm[h]{\ww_h^0}^2
- k \sum_{i=0}^{j-1} \inner[h]{\mm_h^{i+1/2} \times \rr_h^i}{\Ph\heff[\mm_h^{i+1/2}] - \alpha \, d_t\mm_h^{i+1}}.
\end{split}
\end{equation*}
A straightforward application of the discrete Young inequality yields the inequalities
\begin{align*}
\sum_{i=0}^{j-1} \norm[\LL^2(\Omega)]{\Grad\mm_h^{i + 1/2}}
& \le \frac{j}{2}
+ \frac{1}{2} \sum_{i=0}^{j-1} \norm[\LL^2(\Omega)]{\Grad\mm_h^i}^2
+ \frac{1}{4} \norm[\LL^2(\Omega)]{\Grad\mm_h^j}^2,\\
\norm[h]{d_t\mm_h^{i+1}}
&\le
\frac{1}{2}
+ \frac{1}{2} \norm[h]{d_t\mm_h^{i+1}}^2.
\end{align*}
Since $\norm[\LL^{\infty}(\Omega)]{\mm_h^{i+1/2}}\le 1$,
$\norm[h]{\rr_h^i}\le \eps$,
and $\norm[h]{\Ph\heff[\mm_h^{i+1/2}]} \le C \hmin^{-1} \norm[\LL^2(\Omega)]{\Grad\mm_h^{i+1/2}}$
(where $C>0$ depends only on the shape-regularity of $\T_h$),
we obtain that
\begin{equation*}
\begin{split}
& k \sum_{i=0}^{j-1} \inner[h]{\mm_h^{i+1/2} \times \rr_h^i}{\Ph\heff[\mm_h^{i+1/2}] - \alpha \, d_t\mm_h^{i+1}} \\
& \ \le k \sum_{i=0}^{j-1}
\norm[\LL^{\infty}(\Omega)]{\mm_h^{i+1/2}} \norm[h]{\rr_h^i}
\big( \norm[h]{\Ph\heff[\mm_h^{i+1/2}]} + \alpha \norm[h]{d_t\mm_h^{i+1}} \big) \\
& \ \le \eps k \sum_{i=0}^{j-1}
\big(C \hmin^{-1} \norm[\LL^2(\Omega)]{\Grad\mm_h^{i+1/2}} + \alpha \norm[h]{d_t\mm_h^{i+1}} \big) \\
& \ \le \frac{(C \hmin^{-1} + \alpha) \eps j k}{2}
+ \frac{C \eps \hmin^{-1} k}{4} \norm[\LL^2(\Omega)]{\Grad\mm_h^j}^2
+ \frac{C \eps \hmin^{-1} k}{2} \sum_{i=0}^{j-1} \norm[\LL^2(\Omega)]{\Grad\mm_h^i}^2
+ \alpha \eps k \sum_{i=0}^{j-1} \norm[h]{d_t\mm_h^{i+1}}^2.
\end{split}
\end{equation*}
Overall, we thus obtain that
\begin{equation*}
\begin{split}
& \frac{1}{2} \left(1-\frac{C \eps \hmin^{-1} k}{2}\right) \norm[\LL^2(\Omega)]{\Grad\mm_h^j}^2
+ \frac{\tau}{2} \norm[h]{\ww_h^j}^2
+ \alpha (1 - \eps) k \sum_{i=0}^{j-1} \norm[h]{d_t\mm_h^{i+1}}^2 \\
& \quad \le \frac{1}{2} \norm[\LL^2(\Omega)]{\Grad\mm_h^0}^2
+ \frac{\tau}{2} \norm[h]{\ww_h^0}^2
+ \frac{(C \hmin^{-1} + \alpha) \eps j k}{2}
+ \frac{C \eps \hmin^{-1} k}{2} \sum_{i=0}^{j-1} \norm[\LL^2(\Omega)]{\Grad\mm_h^i}^2.
\end{split}
\end{equation*}
Using this estimate, the convergence~\eqref{eq:convergence0}, and
the assumption $\eps = \OO(\hmin)$ as $h,\eps \to 0$,
\eqref{eq:stability_mpsx} can be shown by applying 
the discrete Gronwall lemma
(assuming the discretization parameters to be sufficiently small).
\end{proof}
%%%%%%%%%%%%%%%%%%%%

The boundedness result established in Proposition~\ref{prop:boundedness}
is the starting point to prove Theorem~\ref{thm:convergence}{\rm(ii)}.
We omit the presentation of the proof, since this follows line-by-line
the argument used to show the convergence of Algorithm~\ref{alg:mpsx}.
We only stress that, in the proof of the variational formulation~\eqref{eq:weak:variational}
and the energy inequality~\eqref{eq:weak:energy},
the additional contributions arising from the inexact solution of the nonlinear system
are always uniformly bounded by $\eps$
and therefore vanish in the limit.

%%%%%%%%%%%%%%%%%%%%
\section*{Acknowledgements}
The author wishes to thank
M.\ d'Aquino (Parthenope University of Naples)
for his help with the design of the experiment of Section~\ref{sec:nutation}.
This research has been supported by the Austrian Science Fund (FWF)
through the special research program (SFB) \emph{Taming complexity in partial differential systems}
(grant F65).
%%%%%%%%%%%%%%%%%%%%

%%%%%%%%%%%%%%%%%%%%
\bibliographystyle{acm}
\bibliography{ref}

\begin{thebibliography}{10}

\bibitem{ahpprs2014}
{\sc Abert, C., Hrkac, G., Page, M., Praetorius, D., Ruggeri, M., and Suess,
  D.}
\newblock Spin-polarized transport in ferromagnetic multilayers: {A}n
  unconditionally convergent {FEM} integrator.
\newblock {\em Comput. Math. Appl. 68}, 6 (2014), 639--654.

\bibitem{alouges2008a}
{\sc Alouges, F.}
\newblock A new finite element scheme for {L}andau--{L}ifchitz equations.
\newblock {\em Discrete Contin. Dyn. Syst. Ser. S 1}, 2 (2008), 187--196.

\bibitem{aj2006}
{\sc Alouges, F., and Jaisson, P.}
\newblock Convergence of a finite element discretization for the
  {L}andau--{L}ifshitz equation in micromagnetism.
\newblock {\em Math. Models Methods Appl. Sci. 16}, 2 (2006), 299--316.

\bibitem{akst2014}
{\sc Alouges, F., Kritsikis, E., Steiner, J., and Toussaint, J.-C.}
\newblock A convergent and precise finite element scheme for
  {L}andau--{L}ifschitz--{G}ilbert equation.
\newblock {\em Numer. Math. 128}, 3 (2014), 407--430.

\bibitem{as1992}
{\sc Alouges, F., and Soyeur, A.}
\newblock On global weak solutions for {L}andau--{L}ifshitz equations:
  Existence and nonuniqueness.
\newblock {\em Nonlinear Anal. 18}, 11 (1992), 1071--1084.

\bibitem{bartels2005}
{\sc Bartels, S.}
\newblock Stability and convergence of finite-element approximation schemes for
  harmonic maps.
\newblock {\em SIAM J. Numer. Anal. 43}, 1 (2005), 220--238.

\bibitem{bartels2009}
{\sc Bartels, S.}
\newblock Semi-implicit approximation of wave maps into smooth or convex
  surfaces.
\newblock {\em SIAM J. Numer. Anal. 47}, 5 (2009), 3486--3506.

\bibitem{bartels2015a}
{\sc Bartels, S.}
\newblock Fast and accurate finite element approximation of wave maps into
  spheres.
\newblock {\em ESAIM Math. Model. Numer. Anal. 49}, 2 (2015), 551--558.

\bibitem{bartels2015}
{\sc Bartels, S.}
\newblock {\em Numerical methods for nonlinear partial differential equations},
  vol.~47 of {\em Springer Series in Computational Mathematics}.
\newblock Springer, 2015.

\bibitem{bartels2016}
{\sc Bartels, S.}
\newblock Projection-free approximation of geometrically constrained partial
  differential equations.
\newblock {\em Math. Comp. 85}, 299 (2016), 1033--1049.

\bibitem{bfp2007}
{\sc Bartels, S., Feng, X., and Prohl, A.}
\newblock Finite element approximations of wave maps into spheres.
\newblock {\em SIAM J. Numer. Anal. 46}, 1 (2007), 61--87.

\bibitem{bkp2008}
{\sc Bartels, S., Ko, J., and Prohl, A.}
\newblock Numerical analysis of an explicit approximation scheme for the
  {L}andau--{L}ifshitz--{G}ilbert equation.
\newblock {\em Math. Comp. 77}, 262 (2008), 773--788.

\bibitem{blp2009}
{\sc Bartels, S., Lubich, C., and Prohl, A.}
\newblock Convergent discretization of heat and wave map flows to spheres using
  approximate discrete {L}agrange multipliers.
\newblock {\em Math. Comp. 78}, 267 (2009), 1269--1292.

\bibitem{bp2006}
{\sc Bartels, S., and Prohl, A.}
\newblock Convergence of an implicit finite element method for the
  {L}andau--{L}ifshitz--{G}ilbert equation.
\newblock {\em SIAM J. Numer. Anal. 44}, 4 (2006), 1405--1419.

\bibitem{bmdb1996}
{\sc Beaurepaire, E., Merle, J.-C., Daunois, A., and Bigot, J.-Y.}
\newblock Ultrafast spin dynamics in ferromagnetic nickel.
\newblock {\em Phys. Rev. Lett. 76\/} (1996), 4250--4253.

\bibitem{bffgpprs2014}
{\sc Bruckner, F., Feischl, M., F{\"u}hrer, T., Goldenits, P., Page, M.,
  Praetorius, D., Ruggeri, M., and Suess, D.}
\newblock Multiscale modeling in micromagnetics: {E}xistence of solutions and
  numerical integration.
\newblock {\em Math. Models Methods Appl. Sci. 24}, 13 (2014), 2627--2662.

\bibitem{carbou2001}
{\sc Carbou, G.}
\newblock Thin layers in micromagnetism.
\newblock {\em Math. Models Methods Appl. Sci. 11}, 9 (2001), 1529--1546.

\bibitem{crw2011}
{\sc Ciornei, M.-C., Rub\'{i}, J.~M., and Wegrowe, J.-E.}
\newblock Magnetization dynamics in the inertial regime: {N}utation predicted
  at short time scales.
\newblock {\em Phys. Rev. B 83\/} (2011), 020410.

\bibitem{difratta2020}
{\sc Di~Fratta, G.}
\newblock Micromagnetics of curved thin films.
\newblock {\em Z. Angew. Math. Phys. 71}, 4 (2020), 111.

\bibitem{dpprs2019}
{\sc Di~Fratta, G., Pfeiler, C.-M., Praetorius, D., Ruggeri, M., and Stiftner,
  B.}
\newblock Linear second order {IMEX}-type integrator for the (eddy current)
  {L}andau--{L}ifshitz--{G}ilbert equation.
\newblock {\em IMA J. Numer. Anal. 40}, 4 (2020), 2802--2838.

\bibitem{ft2017}
{\sc Feischl, M., and Tran, T.}
\newblock The {Eddy Current--LLG} equations: {FEM-BEM} coupling and a priori
  error estimates.
\newblock {\em SIAM J. Numer. Anal. 55}, 4 (2017), 1786--1819.

\bibitem{gj1997}
{\sc Gioia, G., and James, R.~D.}
\newblock Micromagnetics of very thin films.
\newblock {\em Proc. Roy. Soc. Lond. A 453}, 1956 (1997), 213--223.

\bibitem{ht2014}
{\sc Hadda, M., and Tilioua, M.}
\newblock On magnetization dynamics with inertial effects.
\newblock {\em J. Engrg. Math. 88\/} (2014), 197--206.

\bibitem{hpprss2019}
{\sc Hrkac, G., Pfeiler, C.-M., Praetorius, D., Ruggeri, M., Segatti, A., and
  Stiftner, B.}
\newblock Convergent tangent plane integrators for the simulation of chiral
  magnetic skyrmion dynamics.
\newblock {\em Adv. Comput. Math. 45}, 3 (2019), 1329--1368.

\bibitem{kw2014}
{\sc Karper, T.~K., and Weber, F.}
\newblock A new angular momentum method for computing wave maps into spheres.
\newblock {\em SIAM J. Numer. Anal. 52}, 4 (2014), 2073--2091.

\bibitem{kw2018}
{\sc Kim, E., and Wilkening, J.}
\newblock Convergence of a mass-lumped finite element method for the
  {L}andau--{L}ifshitz equation.
\newblock {\em Quart. Appl. Math. 76}, 2 (2018), 383--405.

\bibitem{lw2008}
{\sc Lin, F., and Wang, C.}
\newblock {\em The analysis of harmonic maps and their heat flows}.
\newblock World Scientific Publishing Co. Pte. Ltd., Hackensack, NJ, 2008.

\bibitem{mt2016}
{\sc Moumni, M., and Tilioua, M.}
\newblock A finite-difference scheme for a model of magnetization dynamics with
  inertial effects.
\newblock {\em J. Engrg. Math. 100\/} (2016), 95--106.

\bibitem{inertia2020}
{\sc Neeraj, K., Awari, N., Kovalev, S., Polley, D., Hagström, N.~Z.,
  Arekapudi, S.~S., Semisalova, A., Lenz, K., Green, B., Deinert, J.-C.,
  Ilyakov, I., Chen, M., Bawatna, M., Scalera, V., d’Aquino, M., Serpico, C.,
  Hellwig, O., Wegrowe, J.-E., Gensch, M., and Bonetti, S.}
\newblock Inertial spin dynamics in ferromagnets.
\newblock {\em Nat. Phys.\/} (2020).

\bibitem{MUMAG}
{\sc NIST}.
\newblock Micromagnetic modeling activity group website.
\newblock
  \href{http://www.ctcms.nist.gov/~rdm/mumag.html}{http://www.ctcms.nist.gov/$\sim$rdm/mumag.html}.
\newblock Accessed on November 25, 2020.

\bibitem{prs2018}
{\sc Praetorius, D., Ruggeri, M., and Stiftner, B.}
\newblock Convergence of an implicit-explicit midpoint scheme for computational
  micromagnetics.
\newblock {\em Comput.\ Math.\ Appl. 75}, 5 (2018).

\bibitem{tataru2004}
{\sc Tataru, D.}
\newblock The wave maps equation.
\newblock {\em Bull. Amer. Math. Soc. (N.S.) 41}, 2 (2004), 185--204.

\bibitem{wm2016}
{\sc Walowski, J., and M\"unzenberg, M.}
\newblock Perspective: {U}ltrafast magnetism and {TH}z spintronics.
\newblock {\em J. Appl. Phys. 120}, 14 (2016), 140901.

\end{thebibliography}
%%%%%%%%%%%%%%%%%%%%

\end{document}